\documentclass[a4paper]{amsart}
\usepackage{graphicx} 
\usepackage[font=scriptsize]{caption}
\usepackage{soul}
\usepackage{mwe}
\usepackage{multicol}
\usepackage{ragged2e}
 \usepackage{subfigure}
 \usepackage{caption}
 \usepackage{wrapfig}
\usepackage{listings}
\usepackage{tikz}
\usetikzlibrary{shapes.geometric,arrows}
\usepackage{caption}
\usepackage{amsmath,amssymb}
\usepackage{lipsum}
\usepackage[ruled,vlined,linesnumbered]{algorithm2e}
\usepackage{verbatim}
\usepackage{float}
\usepackage{algpseudocode}
\usepackage{csquotes}
\usepackage{textcomp}
\usepackage{mathtools}
\usepackage{fancyvrb}

\DeclareCaptionLabelFormat{algnonumber}{Algorithm}
\newtheorem{REMARK}{Remark}[section]
\newtheorem{CLAIM}{Claim}[section]
\newtheorem{HYPOTHESIS}{Hypothesis}[section]
\newtheorem{THEOREM}{Theorem}[section]
 \newtheorem{CONJECTURE}{Conjecture}[section]
 
 \newtheorem{LEMMA}{Lemma}[section]
\newtheorem{COROLLARY}{Corollary}[section]

\numberwithin{equation}{section}

\tikzstyle{startstop} = [rectangle, rounded corners, text centered, draw=black]
\tikzstyle{process} = [rectangle, minimum width=2cm, minimum height=1cm, text centered, draw=black]
\tikzstyle{decision} = [diamond, aspect=3, minimum width=3cm, minimum height=1cm, text centered, draw=black]
\tikzstyle{compute} = [rectangle, minimum width=2cm, minimum height=1cm, text centered, draw=black]
\tikzstyle{estimate} = [rectangle, rounded corners, minimum width=2cm, minimum height=1cm, text centered, draw=black]
\tikzstyle{arrow} = [thick,->,>=stealth]
\begin{document}

\title[Hadwiger number always upper bounds the chromatic number]
 {Hadwiger number always upper bounds the chromatic number --- 1852--1943 --- A far-reaching generalisation of Guthrie's postulate}

\author{T Srinivasa Murthy}

\thanks{Email address: tsm.iisc@gmail.com\\ \\
This work is inspired and motivated by the work in \cite{sri}. It is evident that the proofs in \cite{sri} and this work do not use advanced algebraic tools, but rather classic algebraic methods. They do, however, demonstrate a failure to recognise the subtleties involved in employing classic algebraic tools over a finite field over many decades to prove some of combinatorial mathematics' most difficult unsolved problems. The methods employed would quite certainly be able to solve a slew of unsolved problems.}






\begin{abstract}
In a simple graph $G$, we prove that the \textit{Hadwiger number}, $h(G)$, of the given graph $G$ always upper bounds the \textit{chromatic number}, $\chi(G)$, of the given graph $G$, that is, $\chi(G) \leq h(G)$. This simply stated problem is one of the fundamental questions in combinatorial mathematics, which was made by Hugo Hadwiger in 1943. Consequently, it  independently verifies the most famous Four-Color Theorem: the case $h(G) = 4$ is equivalent to the Four-Color Theorem, that is, every planar graph is $4$-colourable. In our novel approach, we use algebraic settings over a finite field $\mathbb{Z}_p$. The algebraic setting, in essence, begins with the complete graph with $h(G)$ vertices (which is a minor, $\mathcal{M}$, of the given graph $G$) and iteratively extends to the simple graph $G$. This conjecture has remained elusive, owing to a lack of understanding of the interdependence, particularly the importance of Lemma \ref{l1}, Lemma \ref{l2}, Lemma \ref{l3}, and Lemma \ref{l4} in Section 3. 

\end{abstract}

\maketitle
\section{Introduction}
In 1943, Hadwiger proposed the following conjecture. It is one of the deepest unsolved problems in graph theory, according to Bollob\'{a}s, Catalin, and Erd\H{o}s \cite{bol}.  
\begin{CONJECTURE}[Hadwiger's Conjecture]\cite{had, sey}
For every $t > 0$, every graph with no $K_{t+1}$ minor is $t$-colourable.
\end{CONJECTURE}
The \textit{Hadwiger number}, $h(G)$, is the number of vertices in the largest complete graph to which the simple graph $G$ can be contracted. The \textit{chromatic number}, $\chi(G)$, is the minimum number of colors needed for a \textit{vertex coloring} of a simple graph $G$ (\textit{vertex coloring} of a graph $G$, is a map $f: V(G) \rightarrow \mathcal{K}$, where $\mathcal{K}$ is a set of colors, such that no adjacent vertices are assigned
the same color). In algebraic terms, Hadwiger's conjecture can be stated as follows,
\begin{CONJECTURE}
	In a simple graph $G$, $\chi(G) \leq h(G)$.
\end{CONJECTURE}

Hadwiger's conjecture is well-known as a far-reaching generalisation of the Four-Color Theorem, but it remains an open problem despite the efforts of many over the previous eight decades. Hadwiger \cite{had} proved the conjecture for $h(G) \leq 3$. Kuratowski-Wagner Theorem \cite{kur, wag} provides a planar graph characterization that prohibits certain types of structures. Wagner \cite{wag} proved the most interesting result in 1937, specifically that the case $h(G) = 4$ is equivalent to the Four-Color Theorem. 
\begin{THEOREM}$(\textnormal{Wagner's\ Theorem})$\cite{sey, wag}\\
The Hadwiger number, $h(G)$, of a planar graph is at most four.
\end{THEOREM}
Appel, Haken and Koch \cite{app1, app2} proved the most famous Four-Color Theorem in 1977, which was first asked by Francis Guthrie in 1852 \cite{biggs}, and there was also a significant contribution from Robertson, Sanders, Seymour, and Thomas \cite{rob2} by refining the Four-Color Theorem in light of its history of false proofs, and an attempt to refine the Four-Color Theorem by Gonthieras \cite{gon1, gon2} from the perspective of a programming problem. Robertson, Seymour, and Thomas \cite{rob1} proved the case $h(G) = 5$. For $h(G) \geq 6$, the conjecture is still open. The progress made is extensively discussed in the papers \cite{sun, sey,toft}. The weaker version of Hadwiger's conjecture is studied with an emphasis on exploring and understanding Hadwiger's conjecture. It is stated as follows, and the reader can look up interesting recent results at \cite{del,ser}. Hadwiger's conjecture has a large literature, spanning over 80 years, and several conjectures that follow from showing that $\chi(G) \leq h(G)$ may go unnoticed, owing to inaccessible literature.

\begin{CONJECTURE}[Weak Hadwiger Conjecture]\cite{wood}
	For some constant $c$, every graph with no $K_t$-minor is $ct$-colourable.
\end{CONJECTURE} 

In this paper, we settle Hadwiger's conjecture by proving $\chi(G) \leq h(G)$. In our novel approach, we use algebraic settings over a finite field $\mathbb{Z}_p$. The algebraic setting, in essence, begins with the complete graph with $h(G)$ vertices (which is a minor, $\mathcal{M}$, of the given graph $G$) and iteratively extends to the simple graph $G$. This conjecture has remained elusive, owing to a lack of understanding of the interdependence, particularly the importance of Lemma \ref{l1}, Lemma \ref{l2}, Lemma \ref{l3}, and Lemma \ref{l4} in Section 3.

\begin{REMARK}
The work in \cite{sri} inspired and motivated this work. The proofs of Theorem \ref{th2}, Theorem \ref{th1}, and Theorem \ref{th3} follow from the work in \cite{sri}.
\end{REMARK}
\begin{REMARK}
The reader must note that Lemma \ref{l1}, Lemma \ref{l2}, Lemma \ref{l3}, and Lemma \ref{l4} $($all of which are stated in Section 3$)$ are fundamental to the algebraic settings in this paper and are also essential for the validity of Claim \ref{cl2} $($which is stated in Section 3$)$. In other words, without the validity of Lemma \ref{l1}, Lemma \ref{l2}, Lemma \ref{l3}, and Lemma \ref{l4}, and the fact that the largest complete graph to which the given graph $G$ is contractible is $K_t$ $(t > 0)$ $($also, each simple graph $\mathcal{M}_{i}$ $(q \geq i \geq 1)$, defined in Section 2, has no $K_{t+1}$ minor$)$, Claim \ref{cl2} may face a serious existential crisis $($to the foundation of the algebraic approach$)$.
\end{REMARK}
 
The following is the paper's outline: Section 2 includes all necessary notations and definitions, while Section 3 includes polynomial settings, as well as statements and proofs for all claims and corollaries. The schematic representation in Figure \ref{fig:blockcv} provides a bird's-eye view of the paper’s presentation, and Figure \ref{fig:blockcvv} provides the comprehensive view of the relationships among polynomials that are defined in Section 3.

 	\begin{figure}\centering 
	\begin{tikzpicture}[node distance=1.5cm, auto]

	\node (A1) [process] {Lemma \ref{l1}\par};
	\node (B1) [compute, below of=A1] { Corollary \ref{cr1} \par};	
	\node (A2) [compute, right of=A1,xshift=2cm] { Lemma \ref{l2} \par};
	\node (A3) [compute, right of=A2,xshift=2cm] { Lemma \ref{l3}\par};
	\node (B2) [compute, below of=A3] { Corollary \ref{cr2} \par};	
	\node (A4) [compute, right of=A3,xshift=2cm] { Lemma \ref{l5}\par};
	\node (C1) [compute, right of=A4,xshift=2cm] { Lemma \ref{l6}\par};
	\node (A5) [compute, below of=B1,xshift=7cm, ,yshift=-1cm] { Lemma \ref{l4}\par};
	\node (A6) [compute, below of=A5] { Claim \ref{cl1}\par};
	\node (A7) [compute, below of=A6,text width=4cm,yshift=-2cm] {In order to prove that Claim \ref{cl2} is true (which will be proved in Corollary \ref{cor3}), we have to prove that Hypothesis \ref{h1} is not true (which will be proved in Corollary \ref{cor2}). In order to prove that Hypothesis \ref{h1} is not true, we prove Claim \ref{cl4} and Claim \ref{cl3} \par};
	\node (A9) [compute, below of=A7,xshift=-3cm,yshift=-2cm] { Corollary \ref{cor2}\par};
	\node (A10) [compute, right of=A9,xshift=1cm] { Corollary \ref{cor3}\par};
	\node (A11) [compute, right of=A10,xshift=1cm] { Corollary \ref{cor1}\par};
	\node (A12) [compute, right of=A11,xshift=1cm] { Corollary \ref{cor4}\par};
	\node (A8) [compute, below of=A11] { Corollary \ref{cor5}\par};

	\draw[arrow] (A1) -- (B1);
	\draw[arrow] (B1) -- (A5);
	\draw[arrow] (A2) -- (A5);		
	\draw[arrow] (A3) -- (B2);
	\draw[arrow] (A3) -- (B2);		
	\draw[arrow] (B2) -- (A5);		
	\draw[arrow] (A4) -- (A5);
	\draw[arrow] (C1) -- (A5);		
	\draw[arrow] (A5) -- (A6);		
	\draw[arrow] (A6) -- (A7);
	\draw[arrow] (A7) -- (A9);
	\draw[arrow] (A11) -- (A8);
	\draw[arrow] (A9) -- (A10);
	\draw[arrow] (A10) -- (A11);												\draw[arrow] (A11) -- (A12);												
	\end{tikzpicture}	
	\caption{The presentation in the paper is schematically depicted}
	\label{fig:blockcv}
\end{figure}
\begin{figure}
\begin{flushleft}
\fbox{\parbox[position]{14cm}{

\begin{itemize} 
\item The $t$-color set $\mathcal{K} = \{1, 2, \ldots, t-1, t\}$, prime number $p > 2\Delta^2n, \Delta$ is the maximum degree of the graph $G$ with $n$ vertices. For $q+1 \geq i \geq 0$, proving that a simple graph $\mathcal{M}_i$ is $t$-colorable is exactly equivalent to proving that  $\mathbf{P'}_i(\boldsymbol{v}) \not\equiv 0$ mod $p$,

\begin{equation*}
\mathbf{P}_i(\boldsymbol{v}) = \prod_{v_c \in V(\mathcal{M}_i)}\Big(\prod_{v_d \in N_{\mathcal{M}_i}(v_c)}(v_c - v_d)\prod_{l \in \mathbb{Z}_p\setminus \mathcal{K}}(v_c - l)\Big) \in \mathbb{Z}_p[\boldsymbol{v}].
\end{equation*}

\item The $t$-color set $\mathcal{K}' = \{1, 2, \ldots, t-1 \}\cup\{\beta\}$ ($\beta \in \mathbb{Z}_p\setminus\{0, 1, 2, \ldots, t-1, t\}$),
\begin{equation*}
\hat{\mathbf{P}}_i(\boldsymbol{v}) = \prod_{v_c \in V(\mathcal{M}_i)}\Big(\prod_{v_d \in N_{\mathcal{M}_i}(v_c)}(v_c - v_d)\prod_{l \in \mathbb{Z}_p\setminus \mathcal{K}'}(v_c - l)\Big) \in \mathbb{Z}_p[\boldsymbol{v}].
\end{equation*}

\item Interweaving relations among the polynomials $\mathbf{P}_{i-1}(\boldsymbol{v})$, $\hat{\mathbf{P}}_{i-1}(\boldsymbol{v})$, $\mathbf{S}_{i-1}(\boldsymbol{v})$, $\hat{\mathbf{S}}_{i-1}(\boldsymbol{v})$, $\mathbf{H}_{i-1}(\boldsymbol{v})$, $\mathbf{Q}_{i-1}(\boldsymbol{v})$, $\mathbf{G}(\boldsymbol{v})$, and $\mathbf{K}(\boldsymbol{v})$ defined in Section 3 are shown here, as well as a bird's eye view of the paper's presentation. WLOG, we assume that, $e = v_sv_b$ ($1 \leq s, b \leq n$), is the edge that is either contracted or deleted by an elementary operation $\mathcal{o}_i$ on the simple graph $\mathcal{M}_{i-1}$ to obtain the simple graph $\mathcal{M}_i$. In light of Remark \ref{rm2}, $V(\mathcal{M}_{i-1})= M_1\cup M_2\cup \{v_s\} = V(\mathcal{M}_{i-1}\setminus e)$.

\begin{align*}
\mathbf{P}_{i-1}(\boldsymbol{v}) & = \mathbf{H}_{i-1}(\boldsymbol{v})(v_s-v_b)(v_b-v_s) = \mathbf{S}_{i-1}(\boldsymbol{v})(v_b-v_s) = \mathbf{Q}_{i-1}(\boldsymbol{v})\prod_{l=t+1}^p(v_s-l)(v_b-v_s) \\ & = \frac{\mathbf{G}(\boldsymbol{v})}{\prod_{v_c \in M_1}(v_c-t)}\prod_{l=t+1}^p(v_s-l)(v_b-v_s)\\ \\ 
& = \frac{\prod_{v_c \in M_2}(v_c-\beta)}{\prod_{v_c \in M_1}(v_c-t)\prod_{v_c \in M_2}(v_c-t)}\Big(\frac{\mathbf{G}(\boldsymbol{v})}{\prod_{v_c \in M_2}(v_c-\beta)}\prod_{v_c \in M_2}(v_c-t)\Big)\prod_{l=t+1}^p(v_s-l)(v_b-v_s) \\ \\ 
& = \small{ \frac{\prod_{v_c \in M_2}(v_c-\beta)}{\prod_{v_c \in M_1}(v_c-t)\prod_{v_c \in M_2}(v_c- t)}\Big(\frac{\mathbf{G}(\boldsymbol{v})}{\prod_{v_c \in M_2}(v_c-\beta)}\prod_{v_c \in M_2}(v_c-t)\prod_{l \in \mathbb{Z}_p\setminus \mathcal{K}'}(v_s-l)\Big)\frac{\prod_{l=t+1}^p(v_s-l)(v_b-v_s)}{\prod_{l \in \mathbb{Z}_p\setminus \mathcal{K}'}(v_s-l)}} \\ \\ 
& \Bigg(  \prod_{v_c \in M_2}(v_c- \beta)\ \textnormal{divides}\ \mathbf{G'}(\boldsymbol{v})\ \textnormal{but}\ \prod_{v_c \in M_2}(v_c-\beta)^2\ \textnormal{does not divide}\ \mathbf{G'}(\boldsymbol{v}) \Bigg)\\ 
& \equiv \frac{\prod_{v_c \in M_2}(v_c-\beta)}{\prod_{v_c \in M_1}(v_c-t)\prod_{v_c \in M_2}(v_c-t)}\Big(\mathbf{K}(\boldsymbol{v})\Big)\frac{\prod_{l=t+1}^p(v_s-l)(v_b-v_s)}{\prod_{l \in \mathbb{Z}_p\setminus \mathcal{K}'}(v_s-l)} \\ \\ 
& \equiv \frac{\prod_{v_c \in M_2}(v_c-\beta)}{\prod_{v_c \in M_1}(v_c-t)\prod_{v_c \in M_2}(v_c-t)}\Big(\hat{\mathbf{S}}_{i-1}(\boldsymbol{v})\prod_{v_c \in M_1}(v_c-t)\Big)\frac{\prod_{l=t+1}^p(v_s-l)(v_b-v_s)}{\prod_{l \in \mathbb{Z}_p\setminus \mathcal{K}'}(v_s-l)} \\ \\
& \equiv \frac{\prod_{v_c \in M_2}(v_c-\beta)}{\prod_{v_c \in M_1}(v_c-t)\prod_{v_c \in M_2}(v_c-t)}\Big(\hat{\mathbf{P}}_{i-1}(\boldsymbol{v})\prod_{v_c \in M_1}(v_c-t)\Big)\frac{\prod_{l=t+1}^p(v_s-l)}{\prod_{l \in \mathbb{Z}_p\setminus \mathcal{K}'}(v_s-l)}.
\end{align*}

\item It is apparent from the above interwoven relations that just because there exists an $\boldsymbol{\alpha}$ $\alpha_{v_j} \in \mathcal{K}'$ $(1\leq j\leq n)$  such that $\hat{\mathbf{P}}_{i-1}(\boldsymbol{\alpha})\not\equiv 0$ mod $p$, it does not necessarily imply that there exists a $\boldsymbol{\beta}$ $\beta_{v_j} \in \mathcal{K}$ $(1\leq j\leq n)$  such that $\mathbf{P}_{i-1}(\boldsymbol{\beta})\not\equiv 0$ mod $p$.
 
\end{itemize}  }}
\end{flushleft}
\caption{The comprehensive view of the relationships among polynomials that are defined in Section 3.}
	\label{fig:blockcvv}

\end{figure}
 
\section{Preliminaries}
All graphs considered in this paper are simple graphs, that is, finite, connected and without loops and multiple edges. $K_t$ is a complete graph on $t$ vertices. The $t$-color set is denoted by $\mathcal{K}=\{1, 2, \ldots, t-1, t\}$.  Throughout this paper, a graph $\mathcal{M}$
is called a minor of a graph $G$, if $\mathcal{M}$ is the largest complete graph to which the graph $G$
is contractible by a sequence of applications of the following three elementary
operations: 1. Removal of a vertex. 2. Removal of an edge. 3. Contraction of an edge.\\

We only consider the following two elementary operations because removing a vertex is the same as removing all edges incident to that vertex and deleting the isolated vertex: 1. Removal of an edge. 2. Contraction of an edge. Most importantly, we assume that each elementary operation will remove or contract at most one edge, and that the resulting graph will be a simple graph (isolated vertex will be deleted). \\

Consider a graph $G = (V(G), E(G))$ with $n$ vertices, that is, $V(G) = \{v_1, v_2, \ldots, v_n\}$. Let $G_1$ and $G_2$ be simple graphs. If we say $G_1=G_2$, then $V(G_1)=V(G_2)$ and $E(G_1)=E(G_2)$.  Let $\mathcal{o}_i$ denote an elementary operation. Let $\mathcal{o}_1 < \mathcal{o}_2 < \ldots < \mathcal{o}_q < \mathcal{o}_{q+1}$ be a sequence of $q+1$  elementary operations performed on the graph $G$ to obtain the minor of the graph $G$, that is, $\mathcal{M}$. Let $\mathcal{M}_0 = G <  \mathcal{M}_1 < \mathcal{M}_2 < \ldots < \mathcal{M}_q < \mathcal{M} = \mathcal{M}_{q+1}$ denote a sequence of graphs obtained from $G$ by a sequence of elementary operations $\mathcal{o}_1 < \mathcal{o}_2 < \ldots < \mathcal{o}_q < \mathcal{o}_{q+1}$, note that each graph $\mathcal{M}_i$ ($q+1 \geq i \geq 1$) is a simple graph and is obtained by performing an elementary operation $\mathcal{o}_{i}$ on the simple graph $\mathcal{M}_{i-1}$, and the sequence of operations and graphs is ordered with respect to subscript. In other words, to obtain a simple graph $\mathcal{M}_i$ (which is finite, connected, and without loops and multiple edges), elementary operations are performed as follows: $|V(\mathcal{M}_{i-1})| = |V(\mathcal{M}_{i})| + 1$ if $\mathcal{o}_i$ is an edge contraction operation, and $|V(\mathcal{M}_{i-1})| = |V(\mathcal{M}_{i})|$ or $|V(\mathcal{M}_{i-1})| = |V(\mathcal{M}_{i})| + 1$ if $\mathcal{o}_i$ is an edge removal operation.\\

The neighbourhood of the vertex $v_c$, denoted by $N_G(v_c) =\{v_d:e=v_cv_d \in E(G)\}$, is the set of vertices adjacent to the vertex $v_c$ in $G$. We use the notation $G/e$ to represent edge contraction, which results in a simple graph, and $G\setminus{e}$ to represent edge deletion (see Figure \ref{fig:blockcv1} and Figure \ref{fig:blockcv2} for example), which results in a simple graph or union of two simple graphs. Suppose $e = v_sv_b$ is an edge that is contracted in the simple graph $G$. 
Without loss of generality, we choose the vertex $v_s$ of the graph $G$ for isolation and deletion. Then $G/e$ denotes a simple graph obtained after an edge contraction operation that deletes all edges incident to $v_s$ and the isolated vertex $v_s$ and adds new edges such that the vertex $v_b$ is adjacent to all the vertices in $N_G(v_s)\setminus\{\{N_G(v_s)\cap N_G(v_b)\}\cup\{v_b\}\}$ 
(see Figure \ref{fig:blockcv1} and Figure \ref{fig:blockcv3} for example). The maximum degree of a graph is denoted by $\Delta$. $G[A]$ denotes induced subgraph of the graph $G$, where $A \subseteq V(G)$ (see Figure \ref{fig:blockcv4}, Figure \ref{fig:blockcv5}, for example). And $G[[A]]_l$ denotes induced subgraph of the graph $G$ which is simple and degree of each vertex in the induced subgraph $G[[A]]_l$ is at most $l$, where $A \subseteq V(G)$ (see Figure \ref{fig:blockcv6}, Figure \ref{fig:blockcv7}, Figure \ref{fig:blockcv8}, Figure \ref{fig:blockcv9}, for example). $\mathcal{C}(G^l) = \{A \subseteq V(G): G[[A]]_l \}$ is the set of subsets of $V(G)$ such that $G[[A]]_l$, that is, each induced subgraph $G[[A]]_l$ of the graph $G$ is simple and degree of each vertex in the induced subgraph is at most $l$.\\

Let $p$ $(> 2\Delta^2n)$ be a prime number, and $\mathbb{Z}_p$ is a finite field. Let $\mathbf{F}(v_1, v_{2}, \ldots,  v_r)$ $\in \mathbb{Z}_p[v_1, v_2, \ldots, v_r]$ denote polynomial of $v_1, v_{2}, \ldots, v_r$ over $\mathbb{Z}_p$. Further, by Fermat's theorem, we know that $x^p \equiv x$ mod $p$.\\

Let $\mathbf{F'}(v_1, v_{2}, \ldots,  v_r)$  denote the polynomial obtained after applying the Fermat's theorem if exponent of variables is $\geq p$, that is $v_i^{p} \equiv v_i$ mod $p$ ( $1 \leq i \leq r$ ), in $\mathbf{F}(v_1, v_{2}, \ldots,  v_r)$. And, we can observe that the exponent of each $v_i$ ($1 \leq i \leq r$) in $\mathbf{F'}(v_1, v_{2}, \ldots,  v_r)$ is less than or equal to $p-1$. Further, $\mathbf{F}(v_1, v_{2}, \ldots,  v_r)$($\equiv 0$ mod $p$) is a zero polynomial ($\equiv 0$ mod $p$) if $\mathbf{F}(\gamma_{v_1}, \gamma_{v_{2}}, \ldots, \gamma_{v_r})$ $\equiv 0$ mod $p$ for all $\gamma_{v_1} \in \mathbb{Z}_p, \gamma_{v_{2}} \in \mathbb{Z}_p, \ldots, \gamma_{v_r} \in \mathbb{Z}_p$, in other words, $\mathbf{F}(v_1, v_{2}, \ldots,  v_r)$($\not\equiv 0$ mod $p$) is not a zero polynomial ($\not\equiv 0$ mod $p$) if there is a $r$-tuple $(\gamma_{v_1}, \gamma_{v_{2}}, \ldots, \gamma_{v_r})$ $\gamma_{v_i} \in \mathbb{Z}_p$ ($1 \leq i \leq r$), such that $\mathbf{F}(\gamma_{v_1}, \gamma_{v_{2}}, \ldots, \gamma_{v_r})$ $\not\equiv 0$ mod $p$ . And $\mathbf{F'}(v_1, v_{2}, \ldots,  v_r)$ is a zero polynomial ($\equiv 0$ mod $p$) if there is no non-zero coefficient monomial available after applying the Fermat's theorem if exponent of variables is $\geq p$, that is $v_i^{p} \equiv v_i$ mod $p$ ( $1 \leq i \leq r$ ), in $\mathbf{F}(v_1, v_{2}, \ldots,  v_r)$, in other words, $\mathbf{F'}(v_1, v_{2}, \ldots,  v_r)$ is $\not\equiv 0$ mod $p$ if there is a monomial whose coefficient is $\not\equiv 0$ mod $p$. For the sake of notational simplicity, let $\boldsymbol{v} =(v_1, v_2, \ldots , v_n)$ and $\boldsymbol{\alpha} = (\alpha_{v_1}, \alpha_{v_2}, \ldots, \alpha_{v_n})$.
\begin{REMARK}
	Let $\mathbf{F}(v) = v\prod_{\gamma \in \mathbb{Z}_p\setminus\{0\}}(v-\gamma) \equiv v(v^{p-1}-1) \equiv v^p-v$. Then $\mathbf{F'}(v) \equiv v -v \equiv 0$.
\end{REMARK}
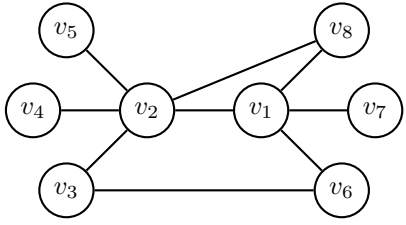
\begin{figure}
\begin{multicols}{2}
{ \centering
 \begin{tikzpicture}[node distance={15mm}, thick, main/.style = {draw, circle}] 
\node[main] (1) {$v_2$}; 
\node[main] (2) [right of=1] {$v_1$};
\node[main] (3) [below left of=1] {$v_3$}; 
\node[main] (4) [left of= 1] {$v_4$}; 
\node[main] (5) [above left of=1] {$v_5$}; 
\node[main] (6) [below right of=2] {$v_6$};  
\node[main] (7) [right of=2] {$v_7$}; 
\node[main] (8) [above right of=2] {$v_8$}; 

\draw (1) -- (2); 
\draw (1) -- (3); 
\draw (1) -- (4); 
\draw (1) -- (5); 
\draw (2) -- (6); 
\draw (2) -- (7); 
\draw (2) -- (8); 
\draw (1) -- (8); 
\draw (3) -- (6);
\end{tikzpicture}
 \captionof{figure}{$H$ is a simple graph.}
 \label{fig:blockcv1}
}
\columnbreak
\begin{flushleft}
$N_H(v_2) = \{v_1, v_3, v_4, v_5, v_8\}$,\\

$N_H(v_2)$$\setminus$$\{\{$$N_H(v_2)$$\cap$$ N_H(v_1)$$\}$$\cup$$\{$$v_1$$\}\}$=$\{v_3, v_4, v_5\}$,\\

$N_{H\setminus e}(v_2)\ or\ N_{H\setminus\{e=v_2v_1\}}(v_2)$=$\{v_3, v_4, v_5, v_8\}$,\\

$B_1$=$\{v_1, v_2, v_4, v_7\}$, $B_2$=$\{v_3, v_5, v_6, v_8\}$, $B_3$=$\{v_1, v_2, v_3, v_6, v_8\}$, $B_4$=$\{v_1, v_2, v_4, v_6,v_7\}$.\\ 

$B_1, B_2, B_3, B_4 \subset$ $V(H)$.
\end{flushleft}
\end{multicols}

\end{figure}

\begin{figure}
\begin{minipage}[t]{0.475\textwidth}
\begin{tikzpicture}[node distance={15mm}, thick, main/.style = {draw, circle}] 
\node[main] (1) {$v_2$}; 
\node[main] (2) [right of=1] {$v_1$};
\node[main] (3) [below left of=1] {$v_3$}; 
\node[main] (4) [left of= 1] {$v_4$}; 
\node[main] (5) [above left of=1] {$v_5$}; 
\node[main] (6) [below right of=2] {$v_6$};  
\node[main] (7) [right of=2] {$v_7$}; 
\node[main] (8) [above right of=2] {$v_8$};

\draw (1) -- (3); 
\draw (1) -- (4); 
\draw (1) -- (5); 
\draw (2) -- (6); 
\draw (2) -- (7); 
\draw (2) -- (8); 
\draw (1) -- (8); 
\draw (3) -- (6);
\end{tikzpicture}
\caption{$H$$\setminus$$e$\ \textnormal{or}\ $H$$\setminus$ $\{$$e$$=$$v_2v_1$$\}$} 
\label{fig:blockcv2}
\end{minipage}
\hfill
\begin{minipage}[t]{0.475\textwidth}
\begin{tikzpicture}[node distance={15mm}, thick, main/.style = {draw, circle}] 

\node[main] (2) [right of=1] {$v_1$};
\node[main] (3) [below left of=1] {$v_3$}; 
\node[main] (4) [left of= 1] {$v_4$}; 
\node[main] (5) [above left of=1] {$v_5$}; 
\node[main] (6) [below right of=2] {$v_6$};  
\node[main] (7) [right of=2] {$v_7$}; 
\node[main] (8) [above right of=2] {$v_8$};

\draw (2) -- (3); 
\draw (2) -- (4); 
\draw (2) -- (5); 
\draw (2) -- (6); 
\draw (2) -- (7); 
\draw (2) -- (8); 
 \draw (3) -- (6);
\end{tikzpicture}
\caption{\textnormal{WLOG, we choose the vertex $v_2$ for isolation and deletion.}\\$H/e\ \textnormal{or}\ H/\{e=v_2v_1\}$.}
\label{fig:blockcv3}
\end{minipage}
\end{figure}

\begin{figure}
\begin{minipage}[t]{0.475\textwidth}
\begin{tikzpicture}[node distance={15mm}, thick, main/.style = {draw, circle}] 
\node[main] (1) {$v_2$}; 
\node[main] (2) [right of=1] {$v_1$};

\node[main] (4) [left of= 1] {$v_4$};

\node[main] (7) [right of=2] {$v_7$};

\draw (1) -- (2); 

\draw (1) -- (4);

\draw (2) -- (7);

\end{tikzpicture}
\caption{$H[B_1]$} 
\label{fig:blockcv4}
\end{minipage}
\hfill
\begin{minipage}[t]{0.475\textwidth}
\begin{tikzpicture}[node distance={15mm}, thick, main/.style = {draw, circle}]

\node[main] (3) [below left of=1] {$v_3$}; 

\node[main] (5) [above left of=1] {$v_5$}; 
\node[main] (6) [below right of=2] {$v_6$};  

\node[main] (8) [above right of=2] {$v_8$};

\draw (3) -- (6);
\end{tikzpicture}
\caption{$H[B_2]$} 
\label{fig:blockcv5}
\end{minipage}
\hfill
\begin{minipage}[t]{0.475\textwidth}
\begin{tikzpicture}[node distance={15mm}, thick, main/.style = {draw, circle}] 
\node[main] (1) {$v_2$}; 
\node[main] (2) [right of=1] {$v_1$};
\node[main] (3) [below left of=1] {$v_3$};

\node[main] (6) [below right of=2] {$v_6$};  

\node[main] (8) [above right of=2] {$v_8$}; 

\draw (1) -- (2); 
\draw (1) -- (3);

\draw (2) -- (6); 

\draw (2) -- (8); 
\draw (1) -- (8); 
\draw (3) -- (6);
\end{tikzpicture}
\caption{$H[[B_3]]_3$}
\label{fig:blockcv6}
\end{minipage}
\hfill
\begin{minipage}[t]{0.475\textwidth}
\begin{tikzpicture}[node distance={15mm}, thick, main/.style = {draw, circle}] 
\node[main] (1) {$v_2$}; 
\node[main] (2) [right of=1] {$v_1$};
\node[main] (4) [left of= 1] {$v_4$};  
\node[main] (6) [below right of=2] {$v_6$};
\node[main] (7) [right of=2] {$v_7$}; 

\draw (1) -- (2); 
\draw (1) -- (4); 
\draw (2) -- (7); 
\draw (2) -- (6);
\end{tikzpicture}
\caption{$H[[B_4]]_3$} 
\label{fig:blockcv7}
\end{minipage}
\hfill
\begin{minipage}[t]{0.475\textwidth}
\begin{tikzpicture}[node distance={15mm}, thick, main/.style = {draw, circle}] 
\node[main] (1) {$v_2$}; 
\node[main] (2) [right of=1] {$v_1$};
\node[main] (3) [below left of=1] {$v_3$}; 
\node[main] (6) [below right of=2] {$v_6$};  
\node[main] (8) [above right of=2] {$v_8$};

\draw (1) -- (3); 
\draw (2) -- (6); 
\draw (2) -- (8); 
\draw (1) -- (8); 
\draw (3) -- (6);
\end{tikzpicture}
\caption{$H[[B_3]]_2$}
\label{fig:blockcv8}
\end{minipage}
\hfill
\begin{minipage}[t]{0.475\textwidth}
\begin{tikzpicture}[node distance={15mm}, thick, main/.style = {draw, circle}] 
\node[main] (1) {$v_2$}; 
\node[main] (2) [right of=1] {$v_1$};
\node[main] (3) [below left of=1] {$v_3$}; 
\node[main] (6) [below right of=2] {$v_6$};  
\node[main] (8) [above right of=2] {$v_8$};

\draw (1) -- (3); 
\draw (2) -- (1); 
\draw (2) -- (8); 
\draw (3) -- (6);
\end{tikzpicture}
\caption{$H[[B_3]]_2$}
\label{fig:blockcv9}
\end{minipage}
\end{figure}

\section{Algebraic settings and results}
We know that a given graph $G$ with $n$ vertices has no $K_{t+1}$ ($t > 0$) minor, and that minor $\mathcal{M}$ (= $\mathcal{M}_{q+1}$) is the complete graph with $t$ vertices, or $h(G) = t$. Since $\mathcal{M}$ is the complete graph, it is a fact that $t$ vertices of $\mathcal{M}$ can be coloured using $t$-colors. To prove Hadwiger's conjecture, we must show that the $n$ vertices of the given graph $G$ can also be coloured using $t$-colors.\\

We know that $\mathcal{M}_i$ ($q+1 \geq i \geq 1$) is a simple graph, and that it can be obtained by performing an elementary operation $\mathcal{o}_{i}$ on $\mathcal{M}_{i-1}$. So, given a graph $G$, $\mathcal{M}_0 = G <  \mathcal{M}_1 < \mathcal{M}_2 < \ldots < \mathcal{M}_q < \mathcal{M} = \mathcal{M}_{q+1}$ is a sequence of graphs constructed from $G$ using a sequence of elementary operations $\mathcal{o}_1 < \mathcal{o}_2 < \ldots < \mathcal{o}_q < \mathcal{o}_{q+1}$.\\

We now define the algebraic settings for colouring the $n$ vertices of the given graph $G$ with $t$-colors as follows ($p > 2\Delta^2n, \Delta$ is the maximum degree of the graph $G$):\\

For $q+1 \geq i \geq 0$,
\begin{equation}
\label{e1}
\mathbf{P}_i(\boldsymbol{v}) = \prod_{v_c \in V(\mathcal{M}_i)}\Big(\prod_{v_d \in N_{\mathcal{M}_i}(v_c)}(v_c - v_d)\prod_{l = t+1}^p(v_c - l)\Big) \in \mathbb{Z}_p[\boldsymbol{v}].
\end{equation}
We notice that proving that the $n$ vertices of a given graph $G$ can be coloured with $t$-colors is exactly equivalent to proving that there exists an $n$-tuple $\boldsymbol{\alpha}$ $\alpha_{v_j} \in \mathcal{K}$ $(1\leq j\leq n)$ such that $\mathbf{P}_0(\boldsymbol{\alpha})\not\equiv 0$ mod $p$. And the vertex colouring of the given graph G with t-colors is defined by the mapping $\ell(v_j)=\alpha_{v_j}$  $(1 \leq j \leq n)$.\\

\begin{REMARK}
	Suppose that there 
	exists an $n$-tuple $\boldsymbol{\alpha}$ such that $\mathbf{P}_0(\boldsymbol{\alpha})\not\equiv 0$ mod $p$. The products $\prod_{v_c \in V(\mathcal{M}_0)}\big(\prod_{l = t+1}^p(v_c - l))$ in $\mathbf{P}_0(\boldsymbol{v})$ then guarantee that $\alpha_{v_j} \in \mathcal{K}$ $(1\leq j\leq n)$.
\end{REMARK}
Now we concentrate on proving the following lemmas, which are fundamental to the algebraic settings in this paper and also necessary for the proof of Claim \ref{cl2} to be valid. In other words, without the validity of Lemmas \ref{l1}, \ref{l2}, \ref{l3} and \ref{l4}, as well as the fact that the largest complete graph to which the given graph $G$ is contractible is $K_t$ ($t > 0$), Claim \ref{cl2} may face an existential crisis (to the foundation of the algebraic approach).

\begin{LEMMA}
	\label{l1}
In a simple graph $H$, let $e = v_sv_b$ be a contracted edge. Suppose $H/e$ is $t$-colorable. The $t$-colors can then be used to colour $H\setminus{e}$.	
\end{LEMMA}
\begin{proof}
In the simple graph $H$, let $e = v_sv_b$ be a contracted edge and $V(H) = \{v_1, v_2, \ldots, v_s,v_b, \ldots, v_h\}$ be the vertex set. We also know that $H/e$ is a simple graph obtained by deleting all edges incident to $v_s$ and the isolated vertex $v_s$ and adding new edges such that the vertex $v_b$ is adjacent to all vertices in $N_H(v_s)\setminus\{\{N_H(v_s)\cap N_H(v_b)\}\cup\{v_b\}\}$, without loss of generality we chose $v_s$ for isolation and deletion, and vertex set $V(H/e) = V(H)\setminus\{v_s\}$. The vertices in $H/e$ are coloured in the following way because the simple graph $H/e$ is $t$-colorable: $\{\ell(v_j): v_j \in V(H/e)\}$. Except for the vertex $v_s$, all of the vertices in $H\setminus{e}$  (it is a simple graph or union of two simple graphs) are coloured with the same colors $\{\ell(v_j): v_j \in V(H/e)\}$ defined for the simple graph $H/e$. As a result, colouring the vertex $v_s$ is sufficient in $H\setminus{e}$, and we colour the vertex $v_s$ with the same colour as the vertex $v_b$ in $H/e$, $\ell(v_s) = \ell(v_b)$. As a result, $H\setminus{e}$ is $t$-colorable.
\end{proof}
\begin{COROLLARY}
	\label{cr1}
In a simple graph $H$, let $e = v_sv_b$ be a contracted edge. Suppose $H/e$ is $t$-colorable. The algebraic setting of the simple graph $($or union of two simple graphs$)$ $H\setminus{e}$ with vertex set $V(H\setminus{e})=\{v_1, v_2, \ldots, v_s,v_b, \ldots, v_h\}$ is the polynomial	
\begin{equation*}
\mathbf{H}(v_1, v_2, \ldots, v_h) = \prod_{v_c \in V(H\setminus{e})}\Big(\prod_{v_d \in N_{H\setminus{e}}(v_c)}(v_c - v_d)\prod_{l = t+1}^p(v_c - l)\Big).
\end{equation*}
Then there exists an $h$-tuple $(\alpha_{v_1}, \alpha_{v_2}, \ldots, \alpha_{v_h})$ $\alpha_{v_j} \in \mathcal{K}$ $(1\leq j\leq h)$ such that $\mathbf{H}(\alpha_{v_1}, \alpha_{v_2}, \ldots, \alpha_{v_h})\not\equiv 0$ mod $p$.	
\end{COROLLARY}
\begin{proof}
Because the simple graph $H\setminus{e}$ is $t$-colorable, as shown by Lemma \ref{l1}, there exists an $h$-tuple $(\alpha_{v_1}, \alpha_{v_2}, \ldots, \alpha_{v_h})$ $\alpha_{v_j} \in \mathcal{K}$ $(1\leq j\leq h)$ such that $\mathbf{H}(\alpha_{v_1}, \alpha_{v_2}, \ldots, \alpha_{v_h})\not\equiv 0$ mod $p$
\end{proof}

\begin{LEMMA}
	\label{l2}
Let $e = v_sv_b$ be an edge that is deleted from a simple graph $H$, resulting in $|V(H\setminus{e})|=|V(H)|$. Suppose that $H\setminus e$ is t-colorable. The algebraic setting of the simple graph $($or union of two simple graphs$)$ $H\setminus{e}$ with vertex set $V(H\setminus{e})=\{v_1, v_2, \ldots, v_s,v_b, \ldots, v_h\}$ is the polynomial	
\begin{equation*}
\mathbf{H}(v_1, v_2, \ldots, v_h) = \prod_{v_c \in V(H\setminus{e})}\Big(\prod_{v_d \in N_{H\setminus{e}}(v_c)}(v_c - v_d)\prod_{l = t+1}^p(v_c - l)\Big).
\end{equation*}
Then there exists an $h$-tuple $(\alpha_{v_1}, \alpha_{v_2}, \ldots, \alpha_{v_h})$ $\alpha_{v_j} \in \mathcal{K}$ $(1\leq j\leq h)$ such that $\mathbf{H}(\alpha_{v_1}, \alpha_{v_2}, \ldots, \alpha_{v_h})\not\equiv 0$ mod $p$.
\end{LEMMA}
\begin{proof}
	This follows from the fact that $H\setminus{e}$ is $t$-colorable.
\end{proof}
\begin{LEMMA}
	\label{l3}
	Let $e = v_sv_b$ be an edge that is deleted from a simple graph $H$, resulting in $|V(H\setminus{e})| + 1 = |V(H)|$. Suppose that $H\setminus{e}$ is $t$-colorable $($isolated vertex will be deleted$)$. The t-colors can then be used to colour $H$.	.
\end{LEMMA}
\begin{proof}
We already know that $|V(H)|=|V(H\setminus{e})|+1$. We assume that $V(H\setminus{e})=V(H) \setminus \{v_s\}$ without losing generality. Because the graph $H\setminus{e}$ is $t$-colorable, we can see that all of the vertices in the graph $H$ are coloured with $t$-colors except the vertex $v_s$; colouring the vertex $v_s$ is sufficient to make $H$ $t$-colorable. We can always colour the vertex $v_s$ such that $H$ is $t$-colorable because the degree of the vertex $v_s$ is one.
\end{proof}
\begin{COROLLARY}
	\label{cr2}
	In a simple graph $H$, let $e = v_sv_b$ is an edge that is removed such that  $|V(H)|=|V(H\setminus{e})|+1$. Suppose graph $H\setminus{e}$ is $t$-colorable $($isolated vertex will be deleted$)$. 
	The algebraic setting of the simple graph $H$ with vertex set $V(H)=\{v_1, v_2, \ldots, v_s,v_b, \ldots, v_h\}$ is the polynomial
	\begin{equation*}
	\mathbf{H}(v_1, v_2, \ldots, v_h) = \prod_{v_c \in V(H)}\Big(\prod_{v_d \in N_{H}(v_c)}(v_c - v_d)\prod_{l = t+1}^p(v_c - l)\Big).
	\end{equation*}
	Then there exists an $h$-tuple $(\alpha_{v_1}, \alpha_{v_2}, \ldots, \alpha_{v_h})$ $\alpha_{v_j} \in \mathcal{K}$ $(1\leq j  \leq h)$ such that $\mathbf{H}(\alpha_{v_1}, \alpha_{v_2}, \ldots, \alpha_{v_h})\not\equiv 0$ mod $p$.	
\end{COROLLARY}
\begin{proof}
The existence of an $h$-tuple $(\alpha_{v_1}, \alpha_{v_2}, \ldots, \alpha_{v_h})$ $\alpha_{v_j} \in \mathcal{K}$ $(1\leq j  \leq h)$ such that $\mathbf{H}(\alpha_{v_1}, \alpha_{v_2}, \ldots, \alpha_{v_h})\not\equiv 0$ mod $p$ follows from Lemma \ref{l3}, which shows that the simple graph H is t-colorable.
\end{proof}

Although the following two lemmas, Lemma \ref{l5} and Lemma \ref{l6}, are based on the fact that the largest complete graph to which the given graph $G$ is contractible is $K_t$ ($t > 0$), their importance cannot be ignored.
 
\begin{LEMMA}
	\label{l5}
Let $\mathcal{Y} = \{Y : Y \subset V(\mathcal{M}_{i-1})$ and $|Y|=t+1\}$. And the polynomial $\mathbf{Y}_{i-1}(v_{y_1}, v_{y_2}, \ldots, v_{y_t}, v_{y_{t+1}})$ is defined as follows, given $Y = \{v_{y_1}, v_{y_2}, \ldots, v_{y_t}, v_{y_{t+1}}\} \in \mathcal{Y}$, $1 \leq y_1, y_2, \ldots, y_t, y_{t+1} \leq n$ $:$ 
$$\mathbf{Y}_{i-1}(v_{y_1}, v_{y_2}, \ldots, v_{y_t}, v_{y_{t+1}}) = \prod_{\substack{1 \leq c < d \leq t+1\\ \text{and}\ e = v_{y_c}v_{y_d} \in E(\mathcal{M}_{i-1})}}(v_{y_c} - v_{y_d})\prod_{v_{y_c} \in Y}\big(\prod_{l = t+1}^p(v_{y_c} - l)\big).$$

Then, given $Y = \{v_{y_1}, v_{y_2}, \ldots, v_{y_t}, v_{y_{t+1}}\} \in \mathcal{Y}$, $\mathbf{Y'}_{i-1}(v_{y_1}, v_{y_2}, \ldots, v_{y_t}, v_{y_{t+1}}) \not\equiv 0$ mod $p$, that is, there exists a  $(\alpha_{v_{y_1}}, \alpha_{v_{y_2}}, \ldots, \alpha_{v_{y_{t+1}}})$ $\alpha_{v_{y_j}} \in \mathcal{K}$ $(1\leq j \leq t+1)$ such that $\mathbf{Y'}_{i-1}(\alpha_{v_{y_1}}, \alpha_{v_{y_2}}, \ldots, \alpha_{v_{y_{t+1}}})\not\equiv 0$ mod $p$.

\end{LEMMA}
\begin{proof}
Suppose that there is a $Y = \{v_{y_1}, v_{y_2}, \ldots, v_{y_t}, v_{y_{t+1}}\} \in \mathcal{Y}$ such that there does not exist  $(\alpha_{v_{y_1}}, \alpha_{v_{y_2}}, \ldots, \alpha_{v_{y_{t+1}}})$  $\alpha_{v_{y_j}} \in \mathcal{K}$ $(1\leq j \leq t+1)$ such that $\mathbf{Y}_{i-1}(\alpha_{v_{y_1}}, $ $\alpha_{v_{y_2}}, \ldots, \alpha_{v_{y_{t+1}}})$ $\not\equiv 0$ mod $p$, that is, $\mathbf{Y'}_{i-1}(v_{y_1}, v_{y_2}, \ldots, v_{y_t}, v_{y_{t+1}}) \equiv 0$ mod $p$. $|\{e=v_{y_c}v_{y_d}\in E(\mathcal{M}_{i-1}): 1 \leq c < d\ \leq t+1\}| = \frac{(t+1)t}{2}$ follows from how polynomial $\mathbf{Y}_{i-1}(v_{y_1}, v_{y_2}, \ldots, v_{y_t}, v_{y_{t+1}})$ is defined. This implies that the vertices in the vertex set $Y$ induce a complete subgraph with $t+1$ vertices, that is, we must have an induced subgraph from the set $Y = \{v_{y_1}, v_{y_2}, \ldots, v_{y_t}, v_{y_{t+1}}\} \in \mathcal{Y}$, that is, any two vertices in $Y$ are adjacent if and only if they are adjacent in the graph $\mathcal{M}_{i-1}$, which is a complete subgraph with $t+1$ vertices. This contradicts the statement that $K_t$ ($t > 0$) is the largest complete graph to which the given graph $G$ can be contracted.
\end{proof}

The following Lemma \ref{l6} is a stronger version of the previous Lemma \ref{l5}, and it will be proven using Brooks' theorem \cite{brook}. The lemma proves that an induced simple graph whose degree is at most $t$ can be coloured using $t$-colors.
\begin{THEOREM}[Brooks Theorem]
Let $G$ be a simple graph. Then $\chi(G) \leq \Delta$ unless $G$ is either a complete graph or an odd cycle.
\end{THEOREM}

\begin{LEMMA}
\label{l6}
The polynomial $\mathbf{Y}_{i-1}(v_{y_1}, v_{y_2}, \ldots, v_{y_k})$ is defined as follows, given $Y = \{v_{y_1}, v_{y_2}, \ldots, v_{y_k}\} \in \mathcal{C}(\mathcal{M}_{i-1}^t)$, $1 \leq y_1, y_2, \ldots, y_k \leq n$$:$ 
$$\mathbf{Y}_{i-1}(v_{y_1}, v_{y_2}, \ldots, v_{y_k}) = \prod_{\substack{1 \leq c < d \leq k\\ \text{and}\ e = v_{y_c}v_{y_d} \in E(\mathcal{M}_{i-1}[[Y]]_t)}}(v_{y_c} - v_{y_d})\prod_{v_{y_c} \in Y}\big(\prod_{l = t+1}^p(v_{y_c} - l)\big).$$

 Then, given $Y = \{v_{y_1}, v_{y_2}, \ldots, v_{y_k}\} \in \mathcal{C}(\mathcal{M}_{i-1}^t)$, $\mathbf{Y'}_{i-1}(v_{y_1}, v_{y_2}, \ldots, v_{y_k}) \not\equiv 0$ mod $p$, that is, there exists a $(\alpha_{v_{y_1}}, \alpha_{v_{y_2}}, \ldots, \alpha_{v_{y_k}})$ $\alpha_{v_{y_j}} \in \mathcal{K}$ $(1\leq j \leq k)$ such that $\mathbf{Y'}_{i-1}(\alpha_{v_{y_1}}, \alpha_{v_{y_2}}, \ldots, \alpha_{v_{y_k}})\not\equiv 0$ mod $p$.
\end{LEMMA}
\begin{proof}
The three cases that follow are as follows:\\

Case 1: $\mathcal{M}_{i-1}[[Y]]_t$ is a complete graph.\\

It follows from the fact that the minor $\mathcal{M}$ of the graph $G$ is $K_t$. Therefore, this case is not possible.\\

Case 2: $\mathcal{M}_{i-1}[[Y]]_t$ is an odd cycle.\\

If this is the case, we know that the minor $\mathcal{M}$ of the graph $G$ is $K_t$ and therefore $t \geq 3$. So, we can conclude that there exists a $(\alpha_{v_{y_1}}, \alpha_{v_{y_2}}, \ldots, \alpha_{v_{y_k}})$ $\alpha_{v_{y_j}} \in \mathcal{K}$ $(1\leq j \leq k)$ such that $\mathbf{Y'}_{i-1}(\alpha_{v_{y_1}}, \alpha_{v_{y_2}}, \ldots, \alpha_{v_{y_k}})\not\equiv 0$ mod $p$.\\

Case 3: $\mathcal{M}_{i-1}[[Y]]_t$ is neither a complete graph nor an odd cycle.\\

$\mathcal{M}_{i-1}[[Y]]_t$ is a simple graph and degree of each vertex in $\mathcal{M}_{i-1}[[Y]]_t$ is at most $t$. It follows from the Brooks Theorem that $\chi(\mathcal{M}_{i-1}[[Y]]_t) \leq t$, and therefore, we can conclude that there exists $(\alpha_{v_{y_1}}, \alpha_{v_{y_2}}, \ldots, \alpha_{v_{y_k}})$ $\alpha_{v_{y_j}} \in \mathcal{K}$ $(1\leq j \leq k)$ such that $\mathbf{Y'}_{i-1}(\alpha_{v_{y_1}}, \alpha_{v_{y_2}}, \ldots, \alpha_{v_{y_k}})\not\equiv 0$ mod $p$.\\
\end{proof}
 	
\begin{REMARK}
	\label{rm1}
Suppose that $e = v_sv_b$ is an edge that is contracted or deleted by an elementary operation $\mathcal{o}_i$ $(q+1 \geq i \geq 1)$ on the simple graph $\mathcal{M}_{i-1}$ to obtain the simple graph $\mathcal{M}_i$, and $\mathcal{M}_{i}$ is $t$-colorable. The polynomial $\mathbf{H}_{i-1}(\boldsymbol{v})$ is defined as follows:
\begin{equation*}
	\mathbf{H}_{i-1}(\boldsymbol{v}) = \prod_{v_c \in V(\mathcal{M}_{i-1}\setminus e)}\Big(\prod_{v_d \in N_{\mathcal{M}_{i-1}\setminus e}(v_c)}(v_c - v_d)\prod_{l = t+1}^p(v_c - l)\Big).
	\end{equation*} 
	Then Lemma \ref{l1} and Lemma \ref{l2} guarantee that there exists an $\boldsymbol{\alpha}$ $\alpha_{v_j} \in \mathcal{K}$ $(1\leq j\leq n)$ such that $\mathbf{H}_{i-1}(\boldsymbol{\alpha})\not\equiv 0$ mod $p$.
\end{REMARK}

In light of Remark \ref{rm1}, the following Lemma \ref{l4} will give us all the hints on how \textit{vertex coloring} of the simple graph $\mathcal{M}_{i-1}$ $(q+1 \geq i \geq 1)$ is interdependent on the Lemma \ref{l1}, Lemma \ref{l2}, and Lemma \ref{l3} and the fact that the largest complete graph to which the given graph $G$ is contractible is $K_t$ ($t > 0$).
\begin{LEMMA}
	\label{l4}
Suppose that an edge $e = v_sv_b$ is contracted or deleted by an elementary operation $\mathcal{o}_i$ $(q+1 \geq i \geq 1)$ on the simple graph $\mathcal{M}_{i-1}$ to obtain the simple graph $\mathcal{M}_i$. 
 The polynomial $\mathbf{S}_{i-1}(\boldsymbol{v})$ is defined as follows, 
 \begin{equation*}
\mathbf{S}_{i-1}(\boldsymbol{v}) = \prod_{v_c \in V(\mathcal{M}_{i-1}\setminus e)}\Big(\prod_{v_d \in N_{\mathcal{M}_{i-1}\setminus e}(v_c)}(v_c - v_d)\prod_{l = t+1}^p(v_c - l)\Big)(v_s-v_b),
\end{equation*}
and there is an $\boldsymbol{\alpha}$ $\alpha_{v_j} \in \mathcal{K}$ $(1\leq j \leq n)$ such that $\mathbf{S}_{i-1}(\boldsymbol{\alpha})\not\equiv 0$ mod $p$. Then $\mathbf{P}_{i-1}(\boldsymbol{\alpha})\not\equiv 0$ mod $p$.
\end{LEMMA}
\begin{proof}
Here we have two cases:\\

\textbf{Case 1}: Suppose that an edge $e = v_sv_b$ is deleted by an elementary operation $\mathcal{o}_i$ on the simple graph $\mathcal{M}_{i-1}$ to obtain the simple graph $\mathcal{M}_i$.\\

Here, again, we have two more cases:\\

\textbf{Case 1a}: when $V(\mathcal{M}_{i-1})=V(\mathcal{M}_i)$.\\

We have, by the definition of $\mathbf{P}_{i-1}(\boldsymbol{v})$, 
\begin{align*}
\mathbf{P}_{i-1}(\boldsymbol{v}) &= \prod_{v_c \in V(\mathcal{M}_{i-1})}\Big(\prod_{v_d \in N_{\mathcal{M}_{i-1}}(v_c)}(v_c - v_d)\prod_{l = t+1}^p(v_c - l)\Big),
\end{align*}
it can be rewritten as,
\begin{align*}
\mathbf{P}_{i-1}(\boldsymbol{v}) &= \prod_{v_c \in V(\mathcal{M}_{i-1}\setminus e)}\Big(\prod_{v_d \in N_{\mathcal{M}_{i-1}\setminus e}(v_c)}(v_c - v_d)\prod_{l = t+1}^p(v_c - l)\Big)(v_s-v_b)(v_b-v_s),
\end{align*}
that is,
\begin{equation*}
\mathbf{P}_{i-1}(\boldsymbol{v}) = \mathbf{S}_{i-1}(\boldsymbol{v})(v_b-v_s).
\end{equation*}
Since $\mathbf{S}_{i-1}(\boldsymbol{\alpha})\not\equiv 0$ mod $p$, we can conclude that,
\begin{equation*}
\mathbf{P}_{i-1}(\boldsymbol{\alpha}) = \mathbf{S}_{i-1}(\boldsymbol{\alpha})(\alpha_{v_b}-\alpha_{v_s}) \not\equiv 0\ \text{mod}\ p.
\end{equation*}	

\textbf{Case 1b}: when $|V(\mathcal{M}_{i-1})|=|V(\mathcal{M}_i)|+1$.\\
Without loss of generality, we assume that, $V(\mathcal{M}_i) = V(\mathcal{M}_{i-1}) \setminus \{v_s\}$. 
We have, by the definition of $\mathbf{P}_{i-1}(\boldsymbol{v})$, 
\begin{align*}
\mathbf{P}_{i-1}(\boldsymbol{v}) &= \prod_{v_c \in V(\mathcal{M}_{i-1})}\Big(\prod_{v_d \in N_{\mathcal{M}_{i-1}}(v_c)}(v_c - v_d)\prod_{l = t+1}^p(v_c - l)\Big),
\end{align*}
it can be rewritten as,
\begin{align*}
\mathbf{P}_{i-1}(\boldsymbol{v})  &= \prod_{v_c \in V(\mathcal{M}_i)}\Big(\prod_{v_d \in N_{\mathcal{M}_i}(v_c)}(v_c - v_d)\prod_{l = t+1}^p(v_c - l)\Big)(v_s-v_b)(v_b-v_s)\prod_{l = t+1}^p(v_s-l).
\end{align*}
From Corollary \ref{cr2}, it follows that there exists an  $\boldsymbol{\alpha}$ $\alpha_{v_j} \in \mathcal{K}$ ($1 \leq j \leq n$) such that $\mathbf{P}_{i-1}(\boldsymbol{\alpha})\not\equiv 0$ mod $p$.\\

\textbf{Case 2}: Suppose that an edge $e = v_sv_b$ is contracted by an elementary operation $\mathcal{o}_i$ on the simple graph $\mathcal{M}_{i-1}$ to obtain the simple graph $\mathcal{M}_i$.\\

We have, by the definition of $\mathbf{P}_{i-1}(\boldsymbol{v})$, 
\begin{align*}
\mathbf{P}_{i-1}(\boldsymbol{v}) &= \prod_{v_c \in V(\mathcal{M}_{i-1})}\Big(\prod_{v_d \in N_{\mathcal{M}_{i-1}}(v_c)}(v_c - v_d)\prod_{l = t+1}^p(v_c - l)\Big),
\end{align*}
it can be rewritten as,
\begin{align*}
\mathbf{P}_{i-1}(\boldsymbol{v}) &= \prod_{v_c \in V(\mathcal{M}_{i-1}\setminus e)}\Big(\prod_{v_d \in N_{\mathcal{M}_{i-1}\setminus e}(v_c)}(v_c - v_d)\prod_{l = t+1}^p(v_c - l)\Big)(v_s-v_b)(v_b-v_s),
\end{align*}
that is,
\begin{equation*}
\mathbf{P}_{i-1}(\boldsymbol{v}) = \mathbf{S}_{i-1}(\boldsymbol{v})(v_b-v_s).
\end{equation*}
Since $\mathbf{S}_{i-1}(\boldsymbol{\alpha})\not\equiv 0$ mod $p$, we can conclude that,
\begin{equation*}
\mathbf{P}_{i-1}(\boldsymbol{\alpha}) = \mathbf{S}_{i-1}(\boldsymbol{\alpha})(\alpha_{v_b}-\alpha_{v_s}) \not\equiv 0\ \text{mod}\ p.
\end{equation*}	
\end{proof}

So far, we have proven all of the necessary lemmas to validate the main result. Now we will concentrate on proving the main point. Since the minor of the given graph $G$, $\mathcal{M}$, is the complete graph with $t$ vertices, we can make the following claim.

\begin{CLAIM}
	\label{cl1}
There exists an $n$-tuple $\boldsymbol{\alpha}$ $\alpha_{v_j} \in \mathcal{K}$ $(1\leq j\leq n)$ such that $\mathbf{P}_{q+1}(\boldsymbol{\alpha})\not\equiv 0$ mod $p$, that is, $\mathbf{P'}_{q+1}(\boldsymbol{v})\not\equiv 0$ mod $p$.
\end{CLAIM}
\begin{proof}
	We know that the minor $\mathcal{M}$ is a complete graph consisting of $t$ vertices. It is a basic fact that the simple graph $\mathcal{M}$ can be colored using $t$-colors. Therefore, there exists an $n$-tuple $\boldsymbol{\alpha}$ $\alpha_{v_j} \in \mathcal{K}$ $(1\leq j \leq n)$ such that the polynomial $\mathbf{P}_{q+1}(\boldsymbol{\alpha})\not\equiv 0$ mod $p$.
\end{proof}
Our goal is achieved by proving the following claim,
\begin{CLAIM}
	\label{cl2}
	There exists an $n$-tuple $\boldsymbol{\alpha}$ $\alpha_{v_j} \in \mathcal{K}$ $(1\leq j\leq n)$ such that $\mathbf{P}_{0}(\boldsymbol{\alpha})\not\equiv 0$ mod $p$, that is, $\mathbf{P'}_{0}(\boldsymbol{v})\not\equiv 0$ mod $p$.
\end{CLAIM}
	We have proven that Claim \ref{cl1} is true, that is, $\mathbf{P'}_{q+1}(\boldsymbol{v})\not\equiv 0$ mod $p$. Suppose there does not exist $(\alpha'_{v_1}, \alpha'_{v_2}, \ldots, \alpha'_{v_n})$ $\alpha'_{v_j} \in \mathcal{K}$ $(1\leq j\leq n)$ such that $\mathbf{P}_{0}(\alpha'_{v_1}, \alpha'_{v_2}, \ldots, \alpha'_{v_n}) \not\equiv 0$ mod $p$. Then there exists some $i$ $(q+1 \geq i \geq 1)$ such that following Hypothesis \ref{h1} has to be true,
	\begin{HYPOTHESIS}
		\label{h1}
		There exists an $\boldsymbol{\alpha}$ $\alpha_{v_j} \in \mathcal{K}$ $(1\leq j \leq n)$ such that $\mathbf{P}_{i}(\boldsymbol{\alpha})\not\equiv 0$ mod $p$ ,that is, $\mathbf{P'}_{i}(\boldsymbol{v})\not\equiv 0$ mod $p$  but there does not exist $(\alpha'_{v_1}, \alpha'_{v_2}, \ldots, \alpha'_{v_n})$ $\alpha'_{v_j} \in \mathcal{K}$ $(1\leq j \leq n)$ such that $\mathbf{P}_{i-1}(\alpha'_{v_1}, \alpha'_{v_2}, \ldots, \alpha'_{v_n})\not\equiv 0$ mod $p$, that is, $\mathbf{P'}_{i-1}(\boldsymbol{v})\equiv 0$ mod $p$.
	\end{HYPOTHESIS} 
\textbf{We have to prove that the Hypothesis \ref{h1} is not true, in other words, we have to prove that $\mathbf{P'}_{i-1}(\boldsymbol{v}) \not\equiv 0$ mod $p$.}\\

Without loss of generality, we assume that, $e = v_sv_b$ ($1 \leq s, b \leq n$), is the edge that is either contracted or deleted by an elementary operation $\mathcal{o}_i$ on the simple graph $\mathcal{M}_{i-1}$ to obtain the simple graph $\mathcal{M}_i$. That is, $\mathcal{M}_{i} = \mathcal{M}_{i-1}/e$ or $\mathcal{M}_{i} = \mathcal{M}_{i-1}\setminus{e}$. 

\begin{REMARK}
\label{rm2}
We ignore the case when $e = v_sv_b$ is an edge that is deleted by an elementary operation $\mathcal{o}_i$ on the simple graph $\mathcal{M}_{i-1}$ to obtain the simple graph $\mathcal{M}_i$ such that  $|V(\mathcal{M}_{i-1})|=|V(\mathcal{M}_i)|+1$ in light of Lemma \ref{l3}. Because Lemma \ref{l3} ensures that there is an $\boldsymbol{\alpha}$ $\alpha_{v_j} \in \mathcal{K}$ $(1\leq j \leq n)$ for which $\mathbf{P}_{i-1}(\boldsymbol{\alpha})\not\equiv 0$ mod $p$. The mapping $\ell(v_j) = \alpha_{v_j}$ $(1 \leq j \leq n)$ will then colour the vertices of the simple graph $\mathcal{M}_{i-1}$, resulting in a $t$-colorable simple graph $\mathcal{M}_{i-1}$. Hypothesis \ref{h1} is therefore false.
\end{REMARK}

Since $\mathbf{P'}_{i}(\boldsymbol{v})\not\equiv 0$ mod $p$, that is, there exists an $\boldsymbol{\alpha}$ $\alpha_{v_j} \in \mathcal{K}$ $(1\leq j \leq n)$ such that $\mathbf{P}_{i}(\boldsymbol{\alpha})\not\equiv 0$ mod $p$. Then the mapping 
$\ell(v_j) = \alpha_{v_j}$ ($1 \leq j \leq n$) will color the vertices of the simple graph $\mathcal{M}_{i}$. Therefore vertices of the simple graph $\mathcal{M}_{i}$ can be colored using $t$-colors. It follows from Remark \ref{rm1} that there exists a $(\alpha'_{v_1}, \alpha'_{v_2}, \ldots, \alpha'_{v_n})$ $\alpha'_{v_j} \in \mathcal{K}$ $(1\leq j \leq n)$ such that $\mathbf{H}_{i-1}(\alpha'_{v_1}, \alpha'_{v_2}, \ldots, \alpha'_{v_n})\not\equiv 0$ mod $p$, that is,

\begin{equation*}
	\mathbf{H}_{i-1}(\boldsymbol{v}) = \prod_{v_c \in V(\mathcal{M}_{i-1}\setminus e)}\Big(\prod_{v_d \in N_{\mathcal{M}_{i-1}\setminus e}(v_c)}(v_c - v_d)\prod_{l = t+1}^p(v_c - l)\Big) \not\equiv 0\ \textnormal{mod}\ p.
	\end{equation*} 

Now, consider the polynomial $\mathbf{S}_{i-1}(\boldsymbol{v})$ (as defined in Lemma \ref{l4}),
$\mathbf{S}_{i-1}(\boldsymbol{v}) = \mathbf{H}_{i-1}(\boldsymbol{v})(v_s-v_b),$ that is,
\begin{equation}
\label{eq21}
\mathbf{S}_{i-1}(\boldsymbol{v}) = \prod_{v_c \in V(\mathcal{M}_{i-1}\setminus e)}\Big(\prod_{v_d \in N_{\mathcal{M}_{i-1}\setminus e}(v_c)}(v_c - v_d)\prod_{l = t+1}^p(v_c - l)\Big)(v_s-v_b).
\end{equation}
 In light of Lemma \ref{l4}, suppose there exists an  $\boldsymbol{\alpha}$ $\alpha_{v_j} \in \mathcal{K}$ $(1\leq j \leq n)$ such that $\mathbf{S}_{i-1}(\boldsymbol{\alpha})\not\equiv 0$ mod $p$, that is, $\mathbf{S'}_{i-1}(\boldsymbol{v}) \not\equiv 0$ mod $p$, then  $\mathbf{P'}_{i-1}(\boldsymbol{v}) \not\equiv 0$ mod $p$, implying that Hypothesis \ref{h1} is not true.\\

Suppose $\mathbf{S'}_{i-1}(\boldsymbol{v}) \equiv 0$ mod $p$ (congruence relation (\ref{eq21})), that is, there does not exist an $\boldsymbol{\alpha}$ $\alpha_{v_j} \in \mathcal{K}$ $(1\leq j \leq n)$ such that $\mathbf{S}_{i-1}(\boldsymbol{\alpha})\not\equiv 0$ mod $p$, in other words, we have
\begin{align*}
\mathbf{S}_{i-1}(\boldsymbol{v}) &= \prod_{v_c \in V(\mathcal{M}_{i-1}\setminus e)}\Big(\prod_{v_d \in N_{\mathcal{M}_{i-1}\setminus e}(v_c)}(v_c - v_d)\prod_{l = t+1}^p(v_c - l)\Big)(v_s-v_b)\ \equiv 0\ \text{mod}\ p,
\end{align*}
it can be rewritten as,
\begin{align*}
\mathbf{S}_{i-1}(\boldsymbol{v}) &= \underbrace{\Big(\prod_{v_c \in V(\mathcal{M}_{i-1}\setminus e)}\Big(\prod_{v_d \in N_{\mathcal{M}_{i-1}\setminus e}(v_c)}(v_c - v_d)\prod_{\substack{l = t+1\\  \text{if}\ c \neq s}}^p(v_c - l)\Big)(v_s-v_b)\Big)}_{\text{We denote the under-brace products by\ } \mathbf{Q}_{i-1}(\boldsymbol{v}).}\Big(\prod_{l = t+1}^p(v_s-l)\Big)\ \equiv 0\ \text{mod}\ p,
\end{align*}
that is,
\begin{align*}
\mathbf{S}_{i-1}(\boldsymbol{v}) &= \mathbf{Q}_{i-1}(\boldsymbol{v})\prod_{l = t+1}^p(v_s-l) \equiv 0\ \text{mod}\ p,
\end{align*}
where 
\begin{align}
\label{eq20}
\mathbf{Q}_{i-1}(\boldsymbol{v}) = \prod_{v_c \in V(\mathcal{M}_{i-1}\setminus e)}\Big(\prod_{v_d \in N_{\mathcal{M}_{i-1}\setminus e}(v_c)}(v_c - v_d)\prod_{\substack{l = t+1\\  \text{if}\ c \neq s}}^p(v_c - l)\Big)(v_s-v_b).
\end{align}
We can easily claim the following,
\begin{CLAIM}
	\label{cl4}
 The polynomial $\mathbf{Q'}_{i-1}(\boldsymbol{v}) \not\equiv 0$ mod $p$, that is, there exists a $(\gamma'_{v_1}, \gamma'_{v_2}, \ldots, \gamma'_{v_n})$ $\gamma'_{v_j} (j \neq s) \in \mathcal{K}$ $(1\leq j\leq n)$ and $\gamma'_{v_s} \in \mathbb{Z}_p$ such that $\mathbf{Q}_{i-1}(\gamma'_{v_1}, \gamma'_{v_2}, \ldots, \gamma'_{v_n})\not\equiv 0$ mod $p$.
\end{CLAIM}
\begin{proof}
We know that $\mathbf{P'}_i(\boldsymbol{v}) \not\equiv 0$ mod $p$, that is, there exists an $\boldsymbol{\alpha}$ $\alpha_{v_j} \in \mathcal{K}$ $(1\leq j \leq n)$ such that $\mathbf{P}_{i}(\boldsymbol{\alpha})\not\equiv 0$ mod $p$. Then the mapping 
$\ell(v_j) = \alpha_{v_j}$ ($1 \leq j \leq n$) will color the vertices of the simple graph $\mathcal{M}_{i}$. This implies we can use the same colors $\{\ell(v_j) = \alpha_j: v_j \in V(\mathcal{M}_i)\}$ to color all the vertices in $\mathcal{M}_{i-1}$ except the vertex $v_s$.

 The claim follows if we are able to color just the vertex $v_s$. And this is possible in the light of the absence of the products $\prod_{l = t+1}^p(v_s-l)$ in $\mathbf{Q}_{i-1}(\boldsymbol{v})$, that is, we have no restriction that the color of the vertex $v_s$ has to be chosen from the $t$-color set $\mathcal{K}$ while coloring the vertex $v_s$ in the simple graph $\mathcal{M}_{i-1}$. Therefore, there exists $(\gamma'_{v_1}, \gamma'_{v_2}, \ldots, \gamma'_{v_n})$ $\gamma'_{v_j} (j \neq s) \in \mathcal{K}$ $(1\leq j\leq n)$ and $\gamma'_{v_s} \in \mathbb{Z}_p$ such that $\mathbf{Q}_{i-1}(\gamma'_{v_1}, \gamma'_{v_2}, \ldots, \gamma'_{v_n})\not\equiv 0$ mod $p$. 	
\end{proof}

\begin{REMARK}
	In this remark, we will understand the reason when $\mathbf{S'}_{i-1}(\boldsymbol{v}) \equiv 0$ mod $p$ $($congruence  relation $($\ref{eq21}$))$.
	
	 From Claim \ref{cl4} we have that $\mathbf{Q'}_{i-1}(\boldsymbol{v}) \not\equiv 0$ mod $p$, is nothing but, 
	\begin{align*}
	\mathbf{Q}_{i-1}(\boldsymbol{v}) = \prod_{v_c \in V(\mathcal{M}_{i-1}\setminus e)}\Big(\prod_{v_d \in N_{\mathcal{M}_{i-1}\setminus e}(v_c)}(v_c - v_d)\prod_{\substack{l = t+1\\  \text{if}\ c \neq s}}^p(v_c - l)\Big)(v_s-v_b) \not\equiv 0\ mod\ p.
	\end{align*}
	And $\mathbf{Q'}_{i-1}(\boldsymbol{v})$ can be written as follows $($after applying Fermat's theorem to each $v_r$ $(\substack{1 \leq r \leq n \\r\neq s})$ in $\mathbf{Q}_{i-1}(\boldsymbol{v})$$)$,
	\begin{equation}
		\label{req1} 
		\mathbf{Q'}_{i-1}(\boldsymbol{v}) \equiv \sum_{\prod_{\substack{r=1\\r\neq s}}^nl_r}\mathcal{C}_{\prod_{\substack{r=1\\r\neq s}}^nl_r}(v_s)\prod_{\substack{r = 1\\r \neq s}}^{n}{v_r}^{l_r},
	\end{equation}
	where $\mathcal{C}_{\prod_{\substack{r=1\\r\neq s}}^nl_r}(v_s)$ $($it is either a univariate polynomial in $v_s$ or a constant$)$ is the coefficient of $\prod_{\substack{r = 1\\r \neq s}}^{n}{v_r}^{l_r}$, the exponent of each $v_r$ $(\substack{1 \leq r \leq n \\r\neq s})$ and $v_s$ is $\leq p-1$.\\
	By the definition of $	\mathbf{S}_{i-1}(\boldsymbol{v})$, we have, 
	\begin{align*}
	\mathbf{S}_{i-1}(\boldsymbol{v}) &= \mathbf{Q}_{i-1}(\boldsymbol{v})\prod_{l = t+1}^p(v_s-l) \equiv 0\ \text{mod}\ p,
	\end{align*}
	$\mathbf{S'}_{i-1}(\boldsymbol{v})$ can also be written as follows using relation $(\ref{req1})$,
	\begin{equation*}
		\mathbf{S'}_{i-1}(\boldsymbol{v}) \equiv \Big(\sum_{\prod_{\substack{r=1\\r\neq s}}^nl_r}\mathcal{C}_{\prod_{\substack{r=1\\r\neq s}}^nl_r}(v_s)\prod_{\substack{r = 1\\r \neq s}}^{n}{v_r}^{l_r}\Big)\prod_{l = t+1}^p(v_s-l),
	\end{equation*}
	\begin{equation*}
		\mathbf{S'}_{i-1}(\boldsymbol{v}) \equiv \sum_{\prod_{\substack{r=1\\r\neq s}}^nl_r}\big(\mathcal{C}_{\prod_{\substack{r=1\\r\neq s}}^nl_r}(v_s)\prod_{l = t+1}^p(v_s-l)\big)\prod_{\substack{r = 1\\r \neq s}}^{n}{v_r}^{l_r},
	\end{equation*}
	where $\mathcal{C}_{\prod_{\substack{r=1\\r\neq s}}^nl_r}(v_s)$ is the coefficient of $\prod_{\substack{r = 1\\r \neq s}}^{n}{v_r}^{l_r}$, the exponent of each $v_r$ $(\substack{1 \leq r \leq n \\r\neq s})$ is $\leq p-1$, and as Fermat's theorem has not been applied to $v_s$ yet, the exponent of $v_s$ is $\leq p+2\Delta-t$.\\
	Further, $\mathbf{S'}_{i-1}(\boldsymbol{v})$ can also be written as follows,
	\begin{equation*}
		\mathbf{S'}_{i-1}(\boldsymbol{v}) \equiv \sum_{\prod_{\substack{r=1\\r\neq s}}^nl_r}\mathcal{D}_{\prod_{\substack{r=1\\r\neq s}}^nl_r}(v_s)\prod_{\substack{r = 1\\r \neq s}}^{n}{v_r}^{l_r},
	\end{equation*}
	where $\mathcal{D}_{\prod_{\substack{r=1\\r\neq s}}^nl_r}(v_s) = \mathcal{C}_{\prod_{\substack{r=1\\r\neq s}}^nl_r}(v_s)\prod_{l = t+1}^p(v_s-l)$ $($it is a univariate polynomial in $v_s$$)$.\\
	We know that $\mathcal{D'}_{\prod_{\substack{r=1\\r\neq s}}^nl_r}(v_s)$ is a polynomial obtained after applying Fermat's theorem to the univariate polynomial $\mathcal{D}_{\prod_{\substack{r=1\\r\neq s}}^nl_r}(v_s)$, using this fact, $\mathbf{S'}_{i-1}(\boldsymbol{v})$ can also be written as follows,
	\begin{equation}
		\label{eqt1}
		\mathbf{S'}_{i-1}(\boldsymbol{v}) \equiv \sum_{\prod_{\substack{r=1\\r\neq s}}^nl_r}\mathcal{D'}_{\prod_{\substack{r=1\\r\neq s}}^nl_r}(v_s)\prod_{\substack{r = 1\\r \neq s}}^{n}{v_r}^{l_r}
	\end{equation}
	Here, we must note that the only possible reason for $\mathbf{S'}_{i-1}(\boldsymbol{v}) \equiv 0 $ mod $p$ from the relation $(\ref{eqt1})$ is that each $\mathcal{D'}_{\prod_{\substack{r=1\\r\neq s}}^nl_r}(v_s) \equiv 0$ mod $p$, that is,
	\begin{equation*}
		\mathbf{S'}_{i-1}(\boldsymbol{v}) \equiv \sum_{\prod_{\substack{r=1\\r\neq s}}^nl_r}\underbrace{\mathcal{D'}_{\prod_{\substack{r=1\\r\neq s}}^nl_r}(v_s)}_{\equiv 0\ \text{mod}\ p}\prod_{\substack{r = 1\\r \neq s}}^{n}{v_r}^{l_r} \equiv 0\ \text{mod}\ p.
	\end{equation*}
	Otherwise, suppose there exists a monomial $\prod_{\substack{r = 1\\r \neq s}}^{n}{v_r}^{l'_r}$ such that its coefficient $\mathcal{D'}_{\prod_{\substack{r=1\\r\neq s}}^nl'_r}(v_s) \not\equiv 0$ mod $p$ in the relation $(\ref{eqt1})$, that is,
	\begin{equation*}
		\mathbf{S'}_{i-1}(\boldsymbol{v}) \equiv \underbrace{\mathcal{D'}_{\prod_{\substack{r=1\\r\neq s}}^nl'_r}(v_s)}_{\not\equiv\ 0\ \text{mod}\ p}\prod_{\substack{r = 1\\r \neq s}}^{n}{v_r}^{l'_r} + \sum_{\prod_{\substack{r=1\\r\neq s}}^nl_r}\mathcal{D'}_{\prod_{\substack{r=1\\r\neq s}}^nl_r}(v_s)\prod_{\substack{r = 1\\r \neq s}}^{n}{v_r}^{l_r},
	\end{equation*}
	so we have that $\mathbf{S'}_{i-1}(\boldsymbol{v}) \not\equiv 0$ mod $p$, to each variable $v_1$, $v_2$, \ldots, $v_n$ in $\mathbf{S'}_{i-1}(\boldsymbol{v})$ we associate the sets $\mathcal{A}_1 = \mathbb{Z}_p,  \mathcal{A}_{2} = \mathbb{Z}_p, \ldots,$ $\mathcal{A}_{n} = \mathbb{Z}_p$ respectively. Then there exists $\gamma_{v_1} \in \mathcal{A}_1, \gamma_{v_2} \in \mathcal{A}_{2}, \ldots,$ $\gamma_{v_n} \in \mathcal{A}_{n}$ such that
	
	\begin{equation*}
		\mathbf{S'}_{i-1}(\gamma_{v_1}, \gamma_{v_2}, \ldots, \gamma_{v_n}) \not\equiv 0\ \textnormal{mod}\ p,
	\end{equation*}
	which implies that Hypothesis \ref{h1} is false.\\
\end{REMARK}

In order to prove that Hypothesis \ref{h1} is false, our strategy is to find a new color set $\mathcal{K}' = \{1, 2, \ldots, t-1\} \cup \{\beta\}$, where $\beta \in \mathbb{Z}_p\setminus\{0, 1, 2, \ldots t-1, t\}$ such that the following Claim \ref{cl3} is true. Later, we use Claim \ref{cl3} to prove that Hypothesis \ref{h1} is false in Corollary \ref{cor2}. We are defining the polynomial $\hat{\mathbf{S}}_{i-1}(\boldsymbol{v})$ as follows, which will be used in the following claim,
\begin{align*}
\hat{\mathbf{S}}_{i-1}(\boldsymbol{v}) &= \prod_{v_c \in V(\mathcal{M}_{i-1}\setminus \{e=v_sv_b\})}\Big(\prod_{v_d \in N_{\mathcal{M}_{i-1}\setminus e}(v_c)}(v_c - v_d)\prod_{l \in \mathbb{Z}_p\setminus\mathcal{K}'}(v_c - l)\Big)(v_s-v_b).
\end{align*}	
\begin{CLAIM}
	\label{cl3}
There exists a $(\beta'_{v_1}, \beta'_{v_2}, \ldots, \beta'_{v_n})$ $\beta'_{v_j} \in \mathcal{K}' = \{1, 2, \ldots, t-1\} \cup \{\beta\}$ $(1\leq j\leq n)$ such that $\hat{\mathbf{S}}_{i-1}(\beta'_{v_1}, \beta'_{v_2}, \ldots, \beta'_{v_n})\not\equiv 0$ mod $p$, that is,
\begin{align*}
\hat{\mathbf{S}}_{i-1}(\boldsymbol{v}) &= \prod_{v_c \in V(\mathcal{M}_{i-1}\setminus \{e=v_sv_b\})}\Big(\prod_{v_d \in N_{\mathcal{M}_{i-1}\setminus e}(v_c)}(v_c - v_d)\prod_{l \in \mathbb{Z}_p\setminus\mathcal{K}'}(v_c - l)\Big)(v_s-v_b)\ \not\equiv 0\ \text{mod}\ p.
\end{align*}	
\end{CLAIM}

Now, we will focus on finding the new color set $\mathcal{K}' = \{1, 2, \ldots, t-1\} \cup \{\beta\}$ where $\beta \in \mathbb{Z}_p\setminus\{0, 1, 2, \ldots t-1, t\}$, using the polynomials $\mathbf{J}_{\prod_{\substack{r=1\\r\neq c}}^nl'_r}(v_c)$ (defined later). To define $\mathbf{J}_{\prod_{\substack{r=1\\r\neq c}}^nl'_r}(v_c)$ we need to define the polynomials $\mathbf{G}(\boldsymbol{v})$ and $\mathbf{H}_{v_c}(\boldsymbol{v})$ as follows, we use the fact that $\mathbf{Q'}_{i-1}(\boldsymbol{v}) \not\equiv 0$ mod $p$ (congruence relation (\ref{eq20}), Claim \ref{cl4}), that is, 
\begin{align}
\label{m7}
\mathbf{Q}_{i-1}(\boldsymbol{v}) = \prod_{v_c \in V(\mathcal{M}_{i-1}\setminus e)}\Big(\prod_{v_d \in N_{\mathcal{M}_{i-1}\setminus e}(v_c)}(v_c - v_d)\prod_{\substack{l = t+1\\  \text{if}\ c \neq s}}^p(v_c - l)\Big)(v_s-v_b) \not\equiv 0\ \text{mod}\ p.
\end{align}
\begin{equation*}
\mathbf{G}(\boldsymbol{v}) = \mathbf{Q}_{i-1}(\boldsymbol{v})\prod_{j=1}^{k}(v_{l_j}-t),
\end{equation*}
where $v_{l_1}, v_{l_2}, \ldots, v_{l_k} \in V(\mathcal{M}_{i-1}\setminus e)\setminus\{v_s\}$ $(1 \leq l_1, l_2, \ldots, l_k \leq n)$ such that 
$\mathbf{G'}(\boldsymbol{v}) \not\equiv 0$ mod $p$ and for every $v_c \in V(\mathcal{M}_{i-1}\setminus e)\setminus\{v_s, v_{l_1}, v_{l_2}, \ldots, v_{l_k}\}$ the polynomial $\mathbf{H'}_{v_c}(\boldsymbol{v}) \equiv 0$ mod $p$, where $\mathbf{H}_{v_c}(\boldsymbol{v})$ is defined as follows
\begin{equation*}
\mathbf{H}_{v_c}(\boldsymbol{v}) = \mathbf{G'}(\boldsymbol{v})(v_c-t),\\
\end{equation*}
note that the exponent of $v_c$ is $\leq p$ and the exponent of each $v_r$ ($\substack{1 \leq r \leq n \\r\neq c})$ is $\leq p-1$ in $\mathbf{H}_{v_c}(\boldsymbol{v})$.

\begin{REMARK}
	In this remark, we understand the case  $V(\mathcal{M}_{i-1}\setminus e)\setminus\{v_s, v_{l_1}, v_{l_2}, \ldots, v_{l_k}\} = \emptyset$, where $v_{l_1}, v_{l_2}, \ldots, v_{l_k} \in V(\mathcal{M}_{i-1}\setminus e)\setminus\{v_s\}$ $(1 \leq l_1, l_2, \ldots, l_k \leq n)$ such that 
	$\mathbf{G'}(\boldsymbol{v}) \not\equiv 0$ mod $p$.	
	We have that $\mathbf{G'}(\boldsymbol{v}) \not\equiv 0$ mod $p$, to each variable $v_1$, $v_2$, \ldots, $v_n$ in $\mathbf{G'}(\boldsymbol{v})$ we associate the sets $\mathcal{A}_1 = \mathbb{Z}_p,  \mathcal{A}_{2} = \mathbb{Z}_p, \ldots,$ $\mathcal{A}_{n} = \mathbb{Z}_p$ respectively. Then there exists $\gamma_{v_1} \in \mathcal{A}_1, \gamma_{v_2} \in \mathcal{A}_{2}, \ldots,$ $\gamma_{v_n} \in \mathcal{A}_{n}$ such that $\mathbf{G}(\gamma_{v_1}, \gamma_{v_2}, \ldots, \gamma_{v_n})\not\equiv 0$ mod $p$. Therefore, the mapping $\ell(v_j)=\gamma_{v_j}$ $($$\substack{1 \leq j\leq n\\ j\neq s}$$)$ defines the coloring of vertices in $V(\mathcal{M}_{i-1})\setminus\{v_s\}$ with only $($$t-1$$)$-colors. And we have the option to color the vertex $v_s$ using the color $t$, which gives the \textit{vertex coloring} of the graph $\mathcal{M}_{i-1}$ using $t$-colors. This implies that Hypothesis \ref{h1} is not true.
\end{REMARK}

Our novel approach involves finding a $\beta$ $(\notin \{0, 1, 2, \ldots, t-1, t\}$) such that for each $v_c \in V(\mathcal{M}_{i-1}\setminus e)\setminus\{v_s, v_{l_1}, v_{l_2}, \ldots, v_{l_k}\}$, $(v_c-\beta)$ is a  square-free factor in $\mathbf{J}_{\prod_{\substack{r=1\\r\neq c}}^nl'_r}(v_c)$ (defined later) and $(v_c-\beta)$ is a  square-free factor in $\mathbf{G'}(\boldsymbol{v})$ as well (that is, $(v_c-\beta)$ divides $\mathbf{J}_{\prod_{\substack{r=1\\r\neq c}}^nl'_r}(v_c)$ but $(v_c-\beta)^2$ does not divide $\mathbf{J}_{\prod_{\substack{r=1\\r\neq c}}^nl'_r}(v_c)$, and this follows from Theorem \ref{th1} stated later). And for each $v_c \in V(\mathcal{M}_{i-1}\setminus e)\setminus\{v_s, v_{l_1}, v_{l_2}, \ldots, v_{l_k}\}$, the square-free factor $(v_c-\beta)$ in $\mathbf{G'}(\boldsymbol{v})$ is replaced by $\big(v_c- t\big)$ to obtain the desired result (this follows from Theorem \ref{th3} stated later). For the sake of notational simplicity, let $M_1 = \{v_{l_1}, v_{l_2}, \ldots, v_{l_k}\}$ and $M_2 = V(\mathcal{M}_{i-1}\setminus e)\setminus\{v_s, v_{l_1}, v_{l_2}, \ldots, v_{l_k}\}$ $(1 \leq l_1, l_2, \ldots, l_k \leq n)$, so $V(\mathcal{M}_{i-1}\setminus e) = M_1 \cup M_2 \cup \{v_s\}$. To choose a new value $\beta$, we are defining polynomial $\mathbf{J}_{\prod_{\substack{r=1\\r\neq c}}^nl'_r}(v_c)$ as follows.\\ 

By the definition of $\mathbf{G}(\boldsymbol{v})$ and $\mathbf{Q}_{i-1}(\boldsymbol{v})$, $\mathbf{G}(\boldsymbol{v})$ can be written as follows,
	\begin{multline}
	\label{m8}
	\mathbf{G}(\boldsymbol{v}) = \prod_{v_c \in V(\mathcal{M}_{i-1}\setminus e)}\Big(\prod_{v_d \in N_{\mathcal{M}_{i-1}\setminus e}(v_c)}(v_c - v_d)\prod_{\substack{l = t+1\\  \text{if}\ c \neq s}}^p(v_c - l)\Big)(v_s-v_b)\prod_{j=1}^{k}(v_{l_j}-t).
	\end{multline}

Given a $v_c \in M_2$, the polynomial $\mathbf{G'}(\boldsymbol{v})$ can be written as (Fermat's theorem is applied to each $v_r$ ($1 \leq r \leq n $)),
\begin{equation}
\label{eq11}
\mathbf{G'}(\boldsymbol{v}) \equiv \sum_{\prod_{\substack{r=1\\r\neq c}}^nl_r}\mathcal{C}_{\prod_{\substack{r=1\\r\neq c}}^nl_r}(v_c)\prod_{\substack{r = 1\\r\neq c}}^n{v_r}^{l_r},
\end{equation}
where $\mathcal{C}_{\prod_{\substack{r=1\\r\neq c}}^nl_r}(v_c)$ is the coefficient of $\prod_{\substack{r = 1\\r\neq c}}^n{v_r}^{l_r}$ and is a univariate polynomial in $v_c$, the exponent of each $v_r$ $(1 \leq r \leq n)$ is $\leq p-1$.\\
Also, we can observe that, given a $v_c \in M_2$, from the relation (\ref{m8}) 
$\mathbf{G'}(\boldsymbol{v})$
can also be written as (excepting $v_c$, Fermat's theorem is applied to each $v_r$ ($\substack{1 \leq r \leq n \\r\neq c})$), 
\begin{equation}
\label{m3}
\mathbf{G'}(\boldsymbol{v}) \equiv  \sum_{\prod_{\substack{r=1\\r\neq c}}^nl_r}\Big(\mathcal{C}^c_{\prod_{\substack{r=1\\r\neq c}}^nl_r}(v_c)\prod_{l=t+1}^p(v_c-l)\Big)\prod_{\substack{r = 1\\r\neq c}}^n{v_r}^{l_r},
\end{equation}
where $0 \leq l_r \leq p-1$, as Fermat's theorem is applied to each $v_r\ (\substack{1 \leq r \leq n \\r\neq c})$ except $v_c$, $\mathcal{C}^c_{\prod_{\substack{r=1\\r\neq c}}^nl_r}(v_c)$ is a univariate polynomial of $v_c$ and exponent of $v_c$ (in the congruence relation (\ref{m3})) in the following products, $$\Big(\mathcal{C}^c_{\prod_{\substack{r=1\\r\neq c}}^nl_r}(v_c)\prod_{l=t+1}^p(v_c-l)\Big)\ \textnormal{is}\ \leq 2\Delta+p-t.$$\\

The most important fact we observe from relation (\ref{eq11}) and relation (\ref{m3}) is that, on applying Fermat's theorem to $v_c$, 
\begin{equation}
\label{eq12}
\mathcal{C}^c_{\prod_{\substack{r=1\\r\neq c}}^nl_r}(v_c)\prod_{l=t+1}^p(v_c-l) \equiv \mathcal{C}_{\prod_{\substack{r=1\\r\neq c}}^nl_r}(v_c),
\end{equation}
as both are coefficient of $\prod_{\substack{r = 1\\r\neq c}}^n{v_r}^{l_r}$.\\ \\

Now, for each $v_c \in M_2$, we define polynomial $\mathbf{J}_{\prod_{\substack{r=1\\r\neq c}}^nl'_r}(v_c)$ using a non-zero coefficient of some monomial $\prod_{\substack{r = 1\\r\neq c}}^n{v_r}^{l'_r}$ in the congruence relation (\ref{m3}) as follows,\\ \\

for each $v_c \in M_2$,
\begin{equation}
\label{eq3}
\mathbf{J}_{\prod_{\substack{r=1\\r\neq c}}^nl'_r}(v_c) = \mathcal{C}^c_{\prod_{\substack{r=1\\r\neq c}}^nl'_r}(v_c)\prod_{l=t+1}^p(v_c-l),
\end{equation}
where, $\mathcal{C}^c_{\prod_{\substack{r=1\\r\neq c}}^nl'_r}(v_c)\prod_{l=t+1}^p(v_c-l)$ is the coefficient 
of some monomial $\prod_{\substack{r=1\\r\neq c}}^nv_r^{l'_r}$ in the congruence relation (\ref{m3}).\\
And $\mathbf{J}_{\prod_{\substack{r=1\\r\neq s}}^nl'_r}(v_s)$ is defined below,
\begin{equation}
\label{eq10}
\mathbf{J}_{\prod_{\substack{r=1\\r\neq s}}^nl'_r}(v_s) = \mathcal{C}^s_{\prod_{\substack{r=1\\r\neq s}}^nl'_r}(v_s),
\end{equation}
where, $\mathcal{C}^s_{\prod_{\substack{r=1\\r\neq s}}^nl'_r}(v_s)$ is the coefficient 
of some monomial $\prod_{\substack{r=1\\r\neq s}}^nv_r^{l'_r}$ in the congruence relation (\ref{m3}), exponent of $v_s$ in $\mathbf{J}_{\prod_{\substack{r=1\\r\neq s}}^nl'_r}(v_s)$ is $\leq 2\Delta$.\\

Now, with the help of the polynomials $\{\mathbf{J}_{\prod_{\substack{r=1\\r\neq c}}^nl'_r}(v_c):v_c \in M_2\} \cup \{\mathbf{J}_{\prod_{\substack{r=1\\r\neq s}}^nl'_r}(v_s)\}$ (as defined in (\ref{eq3}) and (\ref{eq10})), we will find  $\beta$ using the following theorems,

\begin{THEOREM}\cite{sri}
	\label{th2}
	Given a $v_c \in M_2$, $\mathbf{H}_{v_c}(\boldsymbol{v}) = \mathbf{G'}(\boldsymbol{v})(v_c-t)$. For every $v_c \in M_2$, the polynomial $\mathbf{H'}_{v_c}(\boldsymbol{v}) \equiv 0$ mod $p$. Then  
	\begin{equation*}
	\mathbf{G'}(\boldsymbol{v}) \equiv \prod_{v_c \in M_2}\Big(\prod_{l \in \mathbb{Z}_p\setminus\{t\}}(v_c-l)\Big)\bigg(\sum_{\prod_{\substack{r=1\\r \notin K}}^nl_r}\mathcal{C}_{\prod_{\substack{r=1\\r \notin K}}^nl_r}(v_s)\prod_{\substack{r = 1\\r \notin K}}^n{v_r}^{l_r}\bigg),
	\end{equation*}
	where, $K = \{a: v_a \in M_2\} \cup \{s\}$ and $\mathcal{C}_{\prod_{\substack{r=1\\r \notin K}}^nl_r}(v_s)$ is a univariate polynomial in $v_s$.
\end{THEOREM}

\begin{THEOREM}\cite{sri}
	\label{th1}
	There exists a $\beta \in \mathbb{Z}_p\setminus\{0, 1, 2, \ldots, t\}$ such that for each $v_c \in M_2$, $(v_c-\beta)$ divides $\mathbf{J}_{\prod_{\substack{r=1\\r\neq c}}^nl'_r}(v_c)$ but $(v_c-\beta)^2$ does not divide $\mathbf{J}_{\prod_{\substack{r=1\\r\neq c}}^nl'_r}(v_c)$. Moreover, for each $v_c \in M_2$, $(v_c-\beta)$ divides
	$\mathbf{G'}(\boldsymbol{v})$ 
	but $(v_c-\beta)^2$ does not divide $\mathbf{G'}(\boldsymbol{v})$. And also $(v_s-\beta)$ does not divide $\mathbf{J}_{\prod_{\substack{r=1\\r\neq s}}^nl'_r}(v_s)$ and $\mathbf{G'}(\boldsymbol{v})$ 
	as well.
\end{THEOREM}

 \begin{THEOREM}\cite{sri}
	\label{th3}
	\begin{equation*}
	\frac{\mathbf{G}(\boldsymbol{v})}{\prod_{v_c \in M_2}\big(v_c-\beta\big)} \equiv \frac{\mathbf{G'}(\boldsymbol{v})}{\prod_{v_c \in M_2}\big(v_c-\beta\big)},
	\end{equation*}
	that is, 
	\begin{multline*}
	 \prod_{v_c \in M_1} \Big(\prod_{\substack{v_d \in N_{\mathcal{M}_{i-1}\setminus e}(v_c)}}(v_c - v_d)\prod_{l =t}^p(v_c - l)\Big) \\ \prod_{v_c \in M_2} \Big(\prod_{\substack{v_d \in N_{\mathcal{M}_{i-1}\setminus e}(v_c)}}(v_c - v_d)\prod_{l \in \mathbb{Z}_p\setminus\{1, 2, \ldots, t, \beta\}}(v_c - l)\Big) \prod_{\substack{v_d \in N_{\mathcal{M}_{i-1}\setminus e}(v_s)}}(v_s - v_d)\big(v_s - v_b \big) \equiv \\ \prod_{v_c \in M_2}\Big(\prod_{l \in \mathbb{Z}_p\setminus\{t, \beta\}}(v_c - l)\Big)\bigg(\sum_{\prod_{\substack{r=1\\r \notin K}}^nl_r}\mathcal{C}_{\prod_{\substack{r=1\\r \notin K}}^nl_r}(v_s)\prod_{\substack{r = 1\\r \notin K}}^n{v_r}^{l_r}\bigg),
	\end{multline*}
	where, $K = \{a: v_a \in M_2\} \cup \{s\}$ and $\mathcal{C}_{\prod_{\substack{r=1\\r \notin K}}^nl_r}(v_s)$ is a univariate polynomial in $v_s$.\\
\end{THEOREM}
We have now developed all of the necessary theory to have a formal proof of Claim \ref{cl3},

\begin{proof} [Proof of Claim \ref{cl3}] 
Consider the polynomial $\mathbf{K}(\boldsymbol{v})$, which is defined as
	
	\begin{equation*}
	\mathbf{K}(\boldsymbol{v}) =\Bigg(\frac{\mathbf{G'}(\boldsymbol{v})}{\prod_{v_c \in M_2}(v_c-\beta)}\Bigg)\prod_{v_c \in M_2}\big(v_c-t\big)\prod_{l\in \mathbb{Z}_p\setminus\{ 1, 2, \ldots, t-1, \beta\}}(v_s-l).
	\end{equation*}
	
	Using congruence relation in Theorem \ref{th2}, the polynomial $\mathbf{K}(\boldsymbol{v})$ can be rewritten as, 
	
	\begin{multline*}
	\mathbf{K}(\boldsymbol{v}) \equiv  
	\prod_{v_c \in M_2}\Bigg(\Big(\prod_{l \in \mathbb{Z}_p\setminus\{t, \beta\}}(v_c-l)\Big)\Big(v_c-t\Big)\Bigg)\\ \bigg(\sum_{\prod_{\substack{r=1\\r \notin K}}^nl_r}\Big(\mathcal{C}_{\prod_{\substack{r=1\\r \notin K}}^nl_r}(v_s)\prod_{l\in \mathbb{Z}_p\setminus\{ 1, 2, \ldots, t-1, \beta\}}(v_s-l)\Big)\prod_{\substack{r = 1\\r \notin K}}^{n}{v_r}^{l_r}\bigg).
	\end{multline*}\\
	
Now, as explained below, we can conclude that $\mathbf{K'}(\boldsymbol{v}) \not\equiv 0$ mod $p$.	
	 The exponent of each $v_c$  in the following products,\\
	\begin{multline}
	\label{e3}
	\prod_{v_c \in M_2}\Bigg(\prod_{l \in \mathbb{Z}_p\setminus\{t, \beta\}}\big(v_c-l\big)\big(v_c-t\big)\Bigg) = \prod_{v_c \in M_2}\Bigg(\prod_{l \in \mathbb{Z}_p\setminus\{\beta\}}(v_c-l)\Bigg),
	\end{multline}
	is $\leq p-1$.\\
	
	From Theorem \ref{th1}, $v_s - \beta$ does not divide $\mathbf{G'}(\boldsymbol{v})$ and congruence relation (\ref{eq10}) guarantee the existence of a $\mathcal{C}_{\prod_{\substack{r=1\\r \notin K}}^nl'_r}(v_s)$ in 
	$$\bigg(\sum_{\prod_{\substack{r=1\\r \notin K}}^nl_r}\Big(\mathcal{C}_{\prod_{\substack{r=1\\r \notin K}}^nl_r}(v_s)\prod_{l\in \mathbb{Z}_p\setminus\{ 1, 2, \ldots, t-1, \beta\}}(v_s-l)\Big)\prod_{\substack{r = 1\\r \notin K}}^{n}{v_r}^{l_r}\bigg)$$ such that $v_s - \beta$ does not divide $\mathcal{C}_{\prod_{\substack{r=1\\r \notin K}}^nl'_r}(v_s)$.
	Therefore, the following polynomial is not a zero polynomial, after applying Fermat's theorem to $v_s$, that is,\\
	\small{\begin{equation}
	\label{e4}
	\underbrace{\Big(\mathcal{C}_{\prod_{\substack{r=1\\r \notin K}}^nl'_r}(v_s)\prod_{l\in \mathbb{Z}_p\setminus\{ 1, 2, \ldots, t-1, \beta\}}(v_s-l)\Big)}_{\substack{ \not\equiv 0\ \textnormal{mod}\ p\\ (\textnormal{after applying the Fermat's theorem to}\ v_s)}}\prod_{\substack{r = 1\\r \notin K}}^{n}{v_r}^{l'_r} + \sum_{\prod_{\substack{r=1\\r \notin K}}^nl_r}\Big(\mathcal{C}_{\prod_{\substack{r=1\\r \notin K}}^nl_r}(v_s)\prod_{l\in \mathbb{Z}_p\setminus\{ 1, 2, \ldots, t-1, \beta\}}(v_s-l)\Big)\prod_{\substack{r = 1\\r \notin K}}^{n}{v_r}^{l_r} \not\equiv 0\ \textnormal{mod}\ p,
	\end{equation}.}

	and the exponent of each $v_r$ ($r\neq s$) and the exponent of $v_s$ (after applying Fermat's theorem)
	in the congruence relation (\ref{e4}) is $\leq p-1$.\\
	
	So, $\mathbf{K'}(\boldsymbol{v}) \not\equiv 0$ mod $p$ as $M_2 \cap \{\{v_r:r\notin K, 1 \leq r \leq n\} \cup \{v_s\}\} = \emptyset$ and $\mathbf{K'}(\boldsymbol{v})$ is the product of above polynomials (\ref{e3}) and (\ref{e4}).\\
	
	Since $\mathbf{K'}(\boldsymbol{v}) \not\equiv 0$ mod $p$, to each of the variable $v_1, v_{2}, \ldots, v_{n}$, we associate the sets $\mathcal{A}_1 = \mathbb{Z}_p, \mathcal{A}_{2} = \mathbb{Z}_p, \ldots,$ $\mathcal{A}_{n} = \mathbb{Z}_p$ respectively. Then there exists $\beta'_{v_1} \in \mathcal{A}_1, \beta'_{v_2} \in \mathcal{A}_{2}, \ldots,$ $\beta'_{v_n} \in \mathcal{A}_{n}$ such that
	
	\begin{equation}
	\label{eq4}
	\mathbf{K'}(\beta'_{v_1}, \beta'_{v_2}, \ldots, \beta'_{v_n}) \not\equiv 0\ \textnormal{mod}\ p.
	\end{equation}\\
	
	From the congruence relation (\ref{eq4}), we can also conclude that $\mathbf{K'}(\boldsymbol{v}) \not\equiv 0$ mod $p$, and $\mathbf{K'}(\boldsymbol{v})$ can be rewritten as,\\
	
	\begin{multline*}
	\mathbf{K'}(\boldsymbol{v}) \equiv \prod_{v_c \in M_2}\big(v_c-t\big)\Bigg(\prod_{v_c \in M_2}\bigg(\prod_{l \in \mathbb{Z}_p\setminus\{t, \beta\}}(v_c-l)\bigg)\bigg(\sum_{\prod_{\substack{r=1\\r \notin K}}^nl_r}\mathcal{C}_{\prod_{\substack{r=1\\r \notin K}}^nl_r}(v_s)\prod_{\substack{r = 1\\r \notin K}}^{n}{v_r}^{l_r}\bigg)\Bigg) \\ \prod_{l\in \mathbb{Z}_p\setminus\{ 1, 2, \ldots, t-1, \beta\}}(v_s-l).
	\end{multline*}\\
	
	And using Theorem \ref{th3} $\mathbf{K'}(\boldsymbol{v})$ can be rewritten as,
	\begin{multline*}
	\mathbf{K}(\boldsymbol{v}) = \prod_{v_c \in M_2}\big(v_c-t\big)\Bigg( \prod_{v_c \in M_1} \Big(\prod_{v_d \in N_{\mathcal{M}_{i-1}\setminus e}(v_c)}({v_c} - v_d)\prod_{l \in \mathbb{Z}_p\setminus\{1, 2, \ldots, t-1\}}({v_c} - l)\Big)\\\prod_{v_c \in M_2} \Big(\prod_{v_d \in N_{\mathcal{M}_{i-1}\setminus e}(v_c)}({v_c} - v_d)\prod_{l \in \mathbb{Z}_p\setminus\{1, 2, \ldots, t-1, t, \beta\}}({v_c} - l)\Big)\prod_{\substack{v_d \in N_{\mathcal{M}_{i-1}\setminus e}(v_s)}}(v_s - v_d)\big(v_s - v_b\big)\Bigg) \prod_{l\in \mathbb{Z}_p\setminus\{ 1, 2, \ldots, t-1, \beta\}}(v_s-l).
	\end{multline*}\\
	
	So, $\mathbf{K}(\boldsymbol{v})$ is equivalent to the following,
	
	\begin{align}
	\label{e5}
	\mathbf{K}(\boldsymbol{v}) &\equiv \prod_{v_c \in V(\mathcal{M}_{i-1}\setminus e)}\Big(\prod_{v_d \in N_{\mathcal{M}_{i-1}\setminus e}(v_c)}(v_c - v_d)\prod_{l \in \mathbb{Z}_p\setminus\mathcal{K}'}(v_c - l)\Big)\big(v_s - v_b\big),
	\end{align} 
this implies,
\begin{align*}
\mathbf{K}(\boldsymbol{v}) &\equiv 
\hat{\mathbf{S}}_{i-1}(\boldsymbol{v})\not\equiv 0\ \textnormal{mod}\ p \nonumber.
\end{align*}
		From the congruence relation (\ref{eq4}), we have $\mathbf{K}(\beta'_{v_1}, \beta'_{v_2}, \ldots, \beta'_{v_n}) \not\equiv 0$ \textnormal{mod} $p$ ($\beta'_{v_j} \in \mathcal{K'}$), and this implies that $\hat{\mathbf{S}}_{i-1}(\beta'_{v_1}, \beta'_{v_2}, \ldots, \beta'_{v_n}) \not\equiv 0$ \textnormal{mod} $p$ ($\beta'_{v_j} \in \mathcal{K'}$). 	
\end{proof}
\begin{REMARK}
Figure \ref{fig:blockcvv} illustrates that just because there exists an $\boldsymbol{\alpha}$ $\alpha_{v_j} \in \mathcal{K}'$ $(1\leq j\leq n)$  such that $\hat{\mathbf{S}}_{i-1}(\boldsymbol{\alpha})\not\equiv 0$ mod $p$, that does not necessarily imply that there exists a $\boldsymbol{\beta}$ $\beta_{v_j} \in \mathcal{K}$ $(1\leq j\leq n)$  such that $\mathbf{S}_{i-1}(\boldsymbol{\beta})\not\equiv 0$ mod $p$.
\end{REMARK}

\begin{COROLLARY}
	\label{cor2}
	Hypothesis \ref{h1} is not true.
\end{COROLLARY}
\begin{proof}
	We know from Claim \ref{cl3} that there exists a $(\beta'_{v_1}, \beta'_{v_2}, \ldots, \beta'_{v_n})$ $\beta'_{v_j} \in \mathcal{K}'$ $(1\leq j\leq n)$ such that $\hat{\mathbf{S}}_{i-1}(\beta'_{v_1}, \beta'_{v_2}, \ldots, \beta'_{v_n})\not\equiv 0$ mod $p$, that is,
	\begin{align*}
	\hat{\mathbf{S}}_{i-1}(\boldsymbol{v}) &= \prod_{v_c \in V(\mathcal{M}_{i-1}\setminus e)}\Big(\prod_{v_d \in N_{\mathcal{M}_{i-1}\setminus e}(v_c)}(v_c - v_d)\prod_{l \in \mathbb{Z}_p\setminus\mathcal{K}'}(v_c - l)\Big)(v_s-v_b)\ \not\equiv 0\ \text{mod}\ p.
	\end{align*}
	This implies that the following polynomial $\hat{\mathbf{P}}_{i-1}(\boldsymbol{v}) \not\equiv 0\ \text{mod}\ p$, as explained below, 
	\begin{align*}
	\hat{\mathbf{P}}_{i-1}(\boldsymbol{v}) &= \prod_{v_c \in V(\mathcal{M}_{i-1})}\Big(\prod_{v_d \in N_{\mathcal{M}_{i-1}}(v_c)}(v_c - v_d)\prod_{l \in \mathbb{Z}_p\setminus\mathcal{K}'}(v_c - l)\Big),
	\end{align*}
	it can be rewritten as,
	\begin{align*}
	\hat{\mathbf{P}}_{i-1}(\boldsymbol{v}) &= \prod_{v_c \in V(\mathcal{M}_{i-1}\setminus e)}\Big(\prod_{v_d \in N_{\mathcal{M}_{i-1}\setminus e}(v_c)}(v_c - v_d)\prod_{l \in \mathbb{Z}_p\setminus\mathcal{K}'}(v_c - l)\Big)(v_s-v_b)(v_b-v_s),
	\end{align*}
	that is,
	\begin{equation*}
	\hat{\mathbf{P}}_{i-1}(\boldsymbol{v}) = \hat{\mathbf{S}}_{i-1}(\boldsymbol{v})(v_b-v_s).
	\end{equation*}
	Since $\hat{\mathbf{S}}_{i-1}(\beta'_{v_1}, \beta'_{v_2}, \ldots, \beta'_{v_n})\not\equiv 0$ mod $p$, using the Lemma \ref{l4} we can conclude that,
	\begin{equation*}
	\hat{\mathbf{P}}_{i-1}(\beta'_{v_1}, \beta'_{v_2}, \ldots, \beta'_{v_n}) = \hat{\mathbf{S}}_{i-1}(\beta'_{v_1}, \beta'_{v_2}, \ldots, \beta'_{v_n})(\beta'_{v_b}-\beta'_{v_s}) \not\equiv 0\ \text{mod}\ p.
	\end{equation*}	
	Therefore, the mapping 
	$\ell(v_j) = \beta'_{v_j}$ ($1 \leq j \leq n$) ($\beta'_{v_j} \in \mathcal{K'}$) will give \textit{vertex coloring} of the simple graph $\mathcal{M}_{i-1}$.\\
	
On recoloring vertices in $\mathcal{M}_{i-1}$ that are coloured $\beta$ by the colour $t$, we notice that the simple graph $\mathcal{M}_{i-1}$ remains \textit{vertex-colored}. This implies that there exists an $\boldsymbol{\alpha}$ $\alpha_{v_j} \in \mathcal{K}$ $(1\leq j\leq n)$ such that $\mathbf{P}_{i-1}(\boldsymbol{\alpha})\not\equiv 0$ mod $p$, that is,
	\begin{equation*} 
	\mathbf{P}_{i-1}(\boldsymbol{v}) = \prod_{v_c \in V(\mathcal{M}_{i-1})}\Big(\prod_{v_d \in N_{\mathcal{M}_{i-1}}(v_c)}(v_c - v_d)\prod_{l = t+1}^p(v_c - l)\Big) \not\equiv 0\ \text{mod}\ p,
	\end{equation*}
	which implies Hypothesis \ref{h1} is false.
\end{proof}
The Corollary \ref{cor2} has established that Claim \ref{cl2} is true, 
\begin{COROLLARY}
	\label{cor3}
The Claim \ref{cl2} is true, that is, there exists an $n$-tuple $\boldsymbol{\alpha}$ $\alpha_{v_j} \in \mathcal{K}$ $(1\leq j\leq n)$ such that $\mathbf{P}_{0}(\boldsymbol{\alpha})\not\equiv 0$ mod $p$.
\end{COROLLARY}

We can use this fact to claim that $\chi(G) \leq h(G)$ in the following corollary,

 \begin{COROLLARY}
 	\label{cor1}
 	Hadwiger conjecture is true.
 \end{COROLLARY}
\begin{proof}
We know that there exists an $n$-tuple $\boldsymbol{\alpha}$ $\alpha_{v_j} \in \mathcal{K}$ $(1\leq j\leq n)$ such that $\mathbf{P}_{0}(\boldsymbol{\alpha})\not\equiv 0$ mod $p$ because of Claim \ref{cl2}. The mapping $\ell(v_j) = \alpha_{v_j}$ ($1 \leq j \leq n$) will use $t$-colors to colour the vertices of the given simple graph $G$ with $n$ vertices.

\end{proof}
The Corollary \ref{cor1} and Wagner's Theorem together establish that the Four-Color problem is true,
\begin{COROLLARY}
	\label{cor4}
	Every planar graph is 4-colorable.
\end{COROLLARY}	
\begin{proof}
We know from Wagner's Theorem that the \textit{Hadwiger number} of a planar graph is at most $4$. And Corollary 3.5 guarantees that the upper bound of the \textit{chromatic number} of a planar graph is $4$.
\end{proof}
\begin{COROLLARY}
	\label{cor5}
	The Weak Hadwiger Conjecture is true.
\end{COROLLARY}
\begin{proof}
	This follows from Corollary \ref{cor1}.
	\end{proof}

\end{document}